\documentclass[10pt, a4, dvipdfmx]{amsart}
\usepackage{latexsym}
\usepackage{amsmath}
\usepackage{amssymb}
\usepackage{mathrsfs}
\usepackage{graphicx}
\usepackage{color}
\usepackage{pgfpages}
\usepackage{ifthen}
\usepackage{leftidx,tensor}
\usepackage[T1]{fontenc}
\usepackage[latin1]{inputenc}
\usepackage{mathtools}
\usepackage{comment}
\usepackage{dsfont}
\usepackage{ulem}

\usepackage[shortlabels]{enumitem}
\usepackage{aliascnt}
\usepackage[ dvipdfmx,
bookmarks=true, pdfstartview=FitH, pdfborder={0 0 0}, colorlinks=true,citecolor=blue, linkcolor=blue]{hyperref} %
\usepackage{bbm}
\usepackage{nicefrac}

\theoremstyle{plain}
\newtheorem{thm}{Theorem}[section]
\newaliascnt{cor}{thm}
\newaliascnt{prop}{thm}
\newaliascnt{lem}{thm}
\newtheorem{cor}[cor]{Corollary}
\newtheorem{prop}[prop]{Proposition}
\newtheorem{lem}[lem]{Lemma}
\aliascntresetthe{cor}
\aliascntresetthe{prop}
\aliascntresetthe{lem} 
\theoremstyle{definition}
\newaliascnt{defn}{thm}
\newaliascnt{asu}{thm}
\newaliascnt{con}{thm}
\newtheorem{defn}[defn]{Definition}

\aliascntresetthe{defn}
\aliascntresetthe{asu}
\aliascntresetthe{con}
\newcounter{stp}
\newcounter{stpi}
\newcounter{stpci}
\newcounter{stpiii}

\theoremstyle{thm}
\newaliascnt{rem}{thm}
\newaliascnt{exa}{thm}
\newaliascnt{masu}{thm}
\newaliascnt{nota}{thm}
\newaliascnt{sett}{thm}
\newtheorem{rem}[rem]{Remark}

\aliascntresetthe{rem}
\aliascntresetthe{exa}
\aliascntresetthe{masu}
\aliascntresetthe{nota}
\aliascntresetthe{sett}
\numberwithin{equation}{section}

\setlist[enumerate]{font = \normalfont}

\newcommand {\N}	{\mathbb{N}} %
\newcommand {\Z}	{\mathbb{Z}}

\newcommand {\R}	{\mathbb{R}}
\newcommand {\C}	{\mathbb{C}}

\newcommand {\T}	{\mathbb{T}}

\newcommand{\sD}{\mathcal{D}}

\newcommand{\sF}{\mathcal{F}}

\newcommand{\sS}{\mathcal{S}}

\renewcommand{\H}{\mathrm{H}}

\newcommand{\loc}{\mathrm{loc}}

\newcommand{\dk}[1]{\partial_{#1}}
\newcommand{\dt}{\dk{t}} 
\newcommand{\dz}{\dk{z}}

\newcommand{\p}{\partial}

\newcommand{\dH}{\nabla_{\mathrm{H}}}

\newcommand{\xH}{x_\mathrm{H}}

\renewcommand{\div}{\mathrm{div} \, }
\newcommand{\divH}{\mathrm{div}_{\H} \,}

\newcommand{\dBrpq}[3]{\dot{B}^{{#1}}_{{#2}, {#3}}}  %
\newcommand{\LH}[1]{L_{\mathrm{H}}^{#1}}  %
\newcommand{\dBrpqH}[3]{\dot{B}^{{#1}}_{{#2}, {#3}, {\mathrm{H} }}} %
\newcommand{\dBrpqz}[3]{\dot{B}^{{#1}}_{{#2}, {#3}, {z}}} %

\newcommand{\vp}{\varphi}
\newcommand{\vpt}{\widetilde{\varphi}}

\newcommand{\disp}{\displaystyle}

\renewcommand{\exp}[1] { \operatorname{exp}{ \left( #1 \right) } }

\renewcommand{\mid}{\,;\,}

\DeclareMathOperator*{\supp}{supp}

\allowdisplaybreaks[4]
\date{\today}

\title[Uniqueness of weak solutions to the Primitive equations]{Uniqueness of weak solutions to the primitive equations in some anisotropic spaces}
\author{Tim Binz}
\address{Technische Universit\"at Darmstadt\\
        Fachbereich Mathematik\\
        Schlossgartenstr.\ 7\\
        64289 Darmstadt\\
        Germany}
\email{binz@mathematik.tu-darmstadt.de}

\author{Yoshiki Iida}
\address{Department of Pure and Applied Mathematics\\
	Graduate School of Fundamental Science and Engineering\\
	Waseda University\\
	169-8555 Tokyo \\
	Japan}
\email{yoshiki-i737@asagi.waseda.jp}


\begin{document}

\begin{abstract}
We consider the 3D or 2D primitive equations for oceans and atmosphere in the isothermal setting. 
In this paper, we establish a new conditional uniqueness result for weak solutions to the primitive equations, that is, if a weak solution belongs some scaling invariant function spaces, and satisfies some additional assumptions, then the weak solution is unique. In particular, our result can be obtained as different one from $z$-weak solutions framework by adopting some anisotropic approaches with the homogeneous toroidal Besov spaces.
As an application of the proof, we establish the energy equality for weak solutions in the uniqueness class given in the main theorem.
\end{abstract}

\maketitle 

\tableofcontents

\section{Introduction}
\label{sec:intro}

In this paper, we consider the following three-dimensional (full-viscous) primitive equations for oceans and the atmosphere with no Coriolis force in the isothermal setting:
\begin{align}
\left\{
\begin{aligned}\label{eq:PE} 
\dt v + (u\cdot \nabla) v - \Delta v + \dH \Pi &= 0, \\ 
\dz \Pi &= 0, \\ 
\div u &= 0. 
\end{aligned}
\right. \tag{3D PE}
\end{align}
The unknown functions $v: \Omega \to \R^2$, $w: \Omega \to \R$ and $\Pi: \Omega \to \R$ denote the horizontal velocity, the vertical velocity, and the scaler pressure of the fluid, respectively. We write $u={}^\top(v, w)$. 
The operator $\nabla_\H$ denotes the horizontal gradient.
The domain $\Omega$ is described below.
The primitive equations are derived from the (rotating) incompressible Navier--Stokes equations with the hydrostatic balance assumption for the pressure term in the vertical direction.
The rigorous justification of the hydrostatic approximation of the primitive equations are studied in \cite{FGHHKW2020, FGK2021, LT2018}.
\par
The physical boundary problem for the primitive equations \eqref{eq:PE} is usually studied in the layer domain 
\[
D_h \equiv \{ (x_\H, z) \in \R^{2}\times \R \mid -h<z<0 \} \qquad (h>0)
\]
subject to the horizontally periodic condition 
\[
\text{$v$, $w$ and $p$ are $2L$ periodic in $x_\H$ for some $L>0$,}
\]
and the no-vertical flow conditions $w|_{z=-h, 0}=0$. Furthermore, the boundary conditions $\dz v|_{z=-h, 0}=0$ are sometimes imposed. In this setting, extending $v$ (resp.~$w$) evenly (resp.~oddly) in $z$ component with respect to $z=0$ and to be defined on $(-L, L)^{2}\times (-h, h)$, then the extended velocity $v$ (resp.~$w$ is periodic and even (resp.~odd). In this paper, we focus on only periodicity of the velocities $v$ and $w$ and we consider the system \eqref{eq:PE} on $\Omega = \T^2 \times \T = \T^3$. 
\par

We also consider the following two-dimensional primitive equations for the \emph{three}-dimensional geophysical fluid motion:
\begin{align}
\left\{ 
\begin{aligned}\label{eq:PE2}
\dt v + ({}^\top(v^1, w)\cdot \nabla_{x_1, z}) v - \Delta v + (\p_{x_1} \Pi)e_1 &= 0, \\
\dz \Pi &= 0, \\
\p_{x_1} v^1 + \dz w &= 0.
\end{aligned}
\right. \tag{2D PE}
\end{align}
Here, $v: \Omega' \ni (x_1, z)\mapsto (v^1(x_1, z), v^2(x_1, z))\in \R^2$ and $w: \Omega' \ni (x_1, z)\mapsto w(x_1, z)\in\R$ denote the horizontal velocity and the vertical velocity, respectively. The operator $\nabla_{x_1, z}$ is the gradient $\nabla_{x_1, z}={}^\top(\p_{x_1}, \dz)$, and the unit vector $e_1$ is defined as $e_1={}^\top(1, 0)$.
The domain $\Omega'$ is also described as follows.
As in the three-dimensional case, originally we consider \eqref{eq:PE2} in the layer domain
\[
	D_h' \equiv \{ (x_1, z) \in \R\times \R \mid -h<z<0 \} \qquad (h>0)
\]
under the horizontal boundary conditions
\[
	\text{$v$, $w$ and $p$ are $2L$ periodic in $x_1$ for some $L>0$,}
\]
the no-vertical flow conditions $w|_{z=-h, 0}=0$, and the Neumann conditions in $z$ $\dz v|_{z=-h, 0}=0$. After that, extending $v$ and $w$ on $(-L, L)\times (-h, h)$ in $z$ direction as in three-dimensional case, and taking $L=h=\pi$, we consider \eqref{eq:PE2} on $\Omega'=\T^2$.
\par

Lions--Temam--Wang \cite{LTW1992-NewPE, LTW1992-OnLSO, LTW1995} started mathematical studies for the primitive equations, where they showed the existence of global weak solutions in three-dimensional case, by the Galerkin method. Later, Guill\'{e}n Gonz\'{a}lez--Masmoudi--Rodr\'{i}guez Bellido \cite{GMR2001} established the local well-posedness for large initial data in $H^1$ and the global well-posedness for small data in $H^1$. In their seminal work Cao--Titi \cite{CT2007} showed the global well-posedness for the three-dimensional primitive equations for {\it arbitrary large} initial data in $H^1$, by deriving a suitable a priori estimate for the strong solution obtained by \cite{GMR2001}. 
Hieber--Kashiwabara \cite{HK2016} showed the global well-posedness in three dimensional case for large initial data in $H^{2/p, p}$ ($6/5\leq p <\infty$). The assumption $6/5\leq p <\infty$ was relaxed to $1 < p < \infty$ in the work of Hieber--Hussein--Kashiwabara \cite{HHK2016}.
In \cite{HK2016}, they introduced an analytic semigroup (so-called the hydrostatic Stokes semigroup) which yields the solution to the linearized equation of \eqref{eq:PE}, and they proved the local well-posedness according to the method of Kato \cite{K1984}. 
Further properties of the hydrostatic semigroup were investigated in \cite{GGHHK2017, GGHHK2020, GGHHK2021}.
If we take the limit $p\to \infty$ for the initial value space $H^{2/p, p}$ in \cite{HK2016}, then $H^{2/p, p}$ tends to $L^\infty$.
Giga--Gries--Hieber--Hussein--Kashiwabara \cite{GGHHK2021} considered the initial value problem in the scaling invariant space $L^\infty_\H (L^1_z)$, which is larger than $L^\infty$. \par
The existence of the time periodic solutions to the primitive equations for time periodic external forces was discussed in \cite{GHK2017, HS2013, M2010-TPE}. \par
However, the uniqueness of weak solutions to the primitive equations is still an open problem, even for two-dimensional case unlike the Navier--Stokes equations.
Bresch--Guill\'{e}n Gonz\'{a}lez--Masmoudi--Rodr\'{i}guez Bellido \cite{BGMR2003}, Bresch--Kazhikhov--Lemoine \cite{BKL2005}, Petcu \cite{P2007}, and Tachim Medjo \cite{M2010} introduced $z$-weak solution, as a class of weak solutions for which global existence and the uniqueness hold.
Roughly saying, a weak solution $v$ to the primitive equations is said $z$-weak solution if $v$ satisfies the additional regularity 
\[
	\dz v \in L^\infty(0, \infty; L^2) \cap L^2(0, \infty; \dot{H}^1).
\]
In particular, for three-dimensional case, \cite{M2010} showed the global existence and uniqueness of $z$-weak solutions for initial data $v_0\in \{ V \in L^6(\Omega) \mid \dz V \in L^2(\Omega)\}$. Kukavika--Pei--Rusin--Ziane \cite{KPRZ2014} proved the uniqueness for continuous initial data, and Li--Titi \cite{LT2017} proved for initial data $v_0\in \{ V\in L^6(\Omega) \mid \dz V \in L^2(\Omega) \}$ including a small perturbation in $L^\infty$-norm.
\cite{LT2017} also proved that any weak solutions are smooth for $t>0$ by weak-strong uniqueness argument.
Following \cite{LT2017}, Ju \cite{J2021} showed that, for a weak solution $v$ to the two-dimensional primitive equations, if $\dz v$ belongs to $L^4(0, T; L^2)$ or $L^2(0, T; L_{x_1}^{\infty}L^2_z)$, then the uniqueness holds. This result differs from $z$-weak solutions framework. 
The uniqueness problem for weak solutions to the primitive equations was also discussed in \cite{BGMR2003, BKL2005, J2017}. 
\par
Recently, the preprint written by Boutros--Markfelder--Titi \cite{BMT2023} constructed infinitely many generalized weak solutions for the same initial data in $H^1$. Their result also shows the nonuniqueness of generalized weak solutions to the primitive equations. Besides, they treated the hydrostatic Euler equations (as known as the inviscid primitive equations) and two-dimensional Prandtl equations. Their proof followed the convex integration scheme, which is used for solving Onsager's conjecture by Isset \cite{I2018} and Buckmaster--de Lellis--Sz\'{e}kelyhidi--Vicol \cite{BLSV2019}, and for proving nonuniqueness of very weak solutions to the 3D Navier--Stokes euqations by Buckmaster--Vicol \cite{BV2019}.

In this paper, we establish a new class which the uniqueness of weak solutions holds, away from the framework of $z$-weak solutions. We emphasize that in all previous works weak solutions \emph{a-priori} satisfying the energy inequality are considered, whereas this article deals with merely weak solutions which satisfy the energy (in)equality \emph{a posteriori}.
More specifically, we derive the uniqueness by focusing on scale invariance of solution space.
This approach is based on the point of view of Serrin \cite{S1963} for the Navier--Stokes equations. 
Indeed, \cite{KS1996, KT2000, M1984, S1963} showed that weak solutions to the Navier--Stokes equations in a scaling invariant space (so-called the Serrin class) are unique and smooth.
Since the primitive equations essentially have anisotropy for the vertical direction, it is difficult to treat the vertical velocity $w$. 
Indeed, for the three-dimensional case, $w$ is determined by 
\[
w=\int_z^\pi \divH{v} \,d\zeta,
\]
so there is one derivative loss for the horizontal velocity $v$.
Here, the operator $\div_\H$ denotes the two-dimensional horizontal divergence.
In order to overcome the difficulty, it is necessary to use some anisotropic approaches.
Therefore, we introduce homogeneous-type Besov spaces on $\T$, that is, the toroidal Besov spaces defined by Tsurumi \cite{T2019}.
Then, a scaling-invariant condition appears for the uniqueness to hold.
However, when we treat low-frequency part of a weak solution, another condition appears.
For the two-dimensional case, this condition seems to be improved from Ju \cite{J2021}.
In addition, for the three-dimensional case, it seems to be related by the concept of $z$-weak solutions by a formal consideration (see \autoref{rmk:unique} (5) below).
Our precise main result now reads:

\begin{thm}[weak-strong uniqueness]\label{thm:unique}
Let $n=3$ or $2$, and let 
$v_1$ and $v_2$ be weak solutions to the $n$-dimensional primitive equations (\eqref{eq:PE} for $n=3$ or \eqref{eq:PE2} for $n=2$) with the same initial value $v_1(0)=v_2(0)=v_0\in \mathcal{H}$, where the hydrostatic solenoidal subspace $\mathcal{H}$ is given by
\[
	{  \mathcal{H}  }
	\equiv \begin{dcases}
		\left\{
		v\in L^2(\T^3) \mid 
		\text{ $v$ is even in $z$ and  }\ 
		\div_\H\left( \int_{-\pi}^{\pi} v(x_\H, z) \, dz \right) = 0
		\right\} & (n=3), \\
		{  
		\left\{
		v\in L^2(\T^2) \mid 
		\text{ $v$ is even in $z$ and}\ 
		\p_{x_1}\left( \int_{-\pi}^{\pi} v(x_1, z) \, dz \right) = 0
		\right\}  } & { (n=2).  }
	\end{dcases}
\]
Furthermore, suppose that there exists $T=T_{v_1}>0$ such that
\begin{equation}\label{eq:asp}
	v_1 \in L^\beta(0, T; \dBrpqz{-1/p}{p}{\infty}\LH{\infty} )
		\cap L^{\gamma}(0, T; \dBrpqz{2/\gamma}{2}{2}\LH{\infty}),
\end{equation}
where
\begin{equation}\label{eq:scaling invariant}
	\frac{2}{\beta} + \frac{2}{p} = 1 \qquad ( 2 < p,\, \beta < \infty)
\end{equation}
and for some $2<\gamma<\infty$.
 
We also suppose that $v_1$ and $v_2$ satisfy the energy inequality
\begin{equation}\label{eq:energy inequality vj}	
	\frac{1}{2}\| v_j(t) \|_{L^2(\T^n)}^2 + \int_0^t \| \nabla v_j (\tau) \|_{L^2(\T^n)}^2 \,d\tau
	\leq \frac{1}{2}\| v_0 \|_{L^2(\T^n)}^2 \qquad (j=1, 2)
\end{equation}
for a.e.~$t\in (0, \infty)$.
\par
 
Then, $v_1\equiv v_2$ on $[0, T]$  holds  .
\end{thm}

\begin{rem}\label{rmk:unique}
\begin{enumerate}[(1)]
\item The assumption $v_1\in X_T^{\beta, p}\equiv L^\beta(0, T; \dBrpqz{-1/p}{p}{\infty}\LH{\infty} )$ means that $v_1$ belongs to scaling invariant spaces by means of 
	\[
		\| (v_1)_\lambda \|_{X_{T, \lambda}^{\beta, p}}
		= \| v_1 \|_{X_T^{\beta, p}}
		\qquad \text{for $\lambda>0$},
	\]
	where $(v_1)_\lambda(x_1, x_2, z, t) = \lambda v(\lambda x_1, \lambda x_2, \lambda z, \lambda^2 t)$ and $X_{T, \lambda}^{\beta, p, q} = L^\beta(0, T/\lambda^2 ; \dBrpqz{-1/p}{p}{q}(\T_\lambda ; \LH{\infty}({  \T_\lambda^{n-1}  }) ) )$. Indeed, the condition \eqref{eq:scaling invariant} can be immediately rewritten as 
	\[
		\frac{2}{\beta} + \frac{1}{p} + \frac{n-1}{\infty} = 1 + \left( -\frac{1}{p} \right).
	\]

\item We can see from the proof that the uniqueness is valid even if the assumption \eqref{eq:asp} is replaced by
	\[
	v_1 \in L^\beta(0, T; \LH{\infty}\dBrpqz{-1/p}{p}{\infty} )
		\cap L^{\gamma}(0, T; \LH{\infty}\dBrpqz{2/\gamma}{2}{2}).
	\]
	
\item (comparison with the $z$-weak solution framework in the 3D case)
	As we stated, letting $\gamma\to 2+0$, then $L^{\gamma}(0, T; \dBrpqz{2/\gamma}{2}{2}\LH{\infty})$ tends to $L^2(0, T; \dot{H}^1_z\LH{\infty})$.
	For the three dimensional case, if we formally see $\LH{\infty}(\T^2)$ as $\dot{H}^1_\H (\T^2)$ \footnote{Strictly saying, the embedding $\dot{H}^1_\H\hookrightarrow L^\infty_\H$ does NOT hold due to the marginal case of the Sobolev embedding in two dimension, but we see these spaces as similar space with respect to scaling or embedding here.}, the condition 
	\[
	v_1\in L^2(0, T; \dot{H}^1_z\LH{\infty})
	\]
	can be seen as 
	\[
	\dz \dH v_1 \in L^2(0, T; L^2),
	\]
	which corresponds to the definition of $z$-weak solutions.

\item (comparison with Ju \cite{J2021} for the 2D case)
	Letting $\gamma\to 2+0$, then $L^{\gamma}(0, T; \dBrpqz{2/\gamma}{2}{2}\LH{\infty})$ tends to $L^2(0, T; \dot{H}^1_z\LH{\infty})$. This corresponds to the condition
	\[
	\dz v_1 \in L^2(0, T; L^2_z\LH{\infty}),
	\]
	which is one of that derived by Ju \cite{J2021}.
\end{enumerate}
\end{rem}
 
Previous results by \cite{J2021, LT2017} showed the energy equality for weak solutions \emph{satisfying the energy inequality}, whereas we prove the energy equality for weak solutions in slightly larger class than that given in~\eqref{eq:asp} without assuming the energy inequality.

\begin{thm}[energy equality]\label{thm:energy eq}
If a weak solution $v$ with $v(0)=v_0\in   \mathcal{H}  $ to \eqref{eq:PE} satisfies 
\begin{equation}\label{eq:asp EE}
	v \in L^\beta( 0, T; \dBrpqz{-1/p}{p}{\infty}\LH{\infty} )
	\cap L^{\gamma}( 0, T; \dBrpqz{2/\gamma}{2}{\infty}\LH{\infty} )
\end{equation}
for some $T=T_v>0$, where
\[
	\frac{2}{\beta} + \frac{2}{p} = 1 \qquad ( 2 < p,\, \beta < \infty),
\]
and for some $2<\gamma<\infty$. \par
Then, the weak solution $v$ satisfies the following energy equality: 
\begin{equation}\label{eq:energy equality}
	\frac{1}{2}\| v(t) \|_{L^2}^2 
	+ \int_0^t \| \nabla v (\tau) \|_{L^2}^2 \,d\tau 
	= \frac{1}{2}\| v_0 \|_{L^2}^2
\end{equation}
for $0<t\leq T$.
\end{thm}

\begin{rem}
We can see from the proof that the energy equality \eqref{eq:energy equality} is valid even if the assumption \eqref{eq:asp EE} is replaced by
\[
	v \in L^\beta( 0, T; \LH{\infty}\dBrpqz{-1/p}{p}{\infty} )
	\cap L^{\gamma}( 0, T; \LH{\infty}\dBrpqz{2/\gamma}{2}{\infty} ).
\]
\end{rem}

{The proof of the energy equality \eqref{eq:energy equality} is an application of the technique essentially used in the proof of the uniqueness (\autoref{thm:unique}).  } We will prove \autoref{thm:energy eq} in \autoref{sec:proof of EE}.\par

Combining \autoref{thm:unique} and \autoref{thm:energy eq} we obtain uniqueness of strong solutions. More precisely, we obtain
\begin{cor}[strong uniqueness]\label{cor:strong uniqueness}
Let $v_1,v_2 \in C_w([0, \infty); \mathcal{H}) \cap L^2_\loc([0, \infty); (H^1\cap \mathcal{H}))$ with the same initial value $v_1(0)=v_2(0)=v_0\in \mathcal{H}$ be weak solutions to \eqref{eq:PE} or \eqref{eq:PE2}.
Suppose that there exists $T=T_{v_1,v_2}>0$ such that
\begin{equation}\label{eq:asp2}
	v_1, v_2 \in L^\beta(0, T; \dBrpqz{-1/p}{p}{\infty}\LH{\infty} )
		\cap L^{\gamma}(0, T; \dBrpqz{2/\gamma}{2}{2}\LH{\infty}),
\end{equation}
where \eqref{eq:scaling invariant} and for some $2<\gamma<\infty$. \par
Then, $v_1\equiv v_2$ on $[0, T]$ holds.
\end{cor}

\section{Preliminaries}
\label{sec:prelim}

\subsection{Definition and important properties of the toroidal Besov spaces}
\label{subsec:toroidal Besov}

{ 
In this \autoref{subsec:toroidal Besov}, we mainly refer \cite{T2019}, and we also refer Schmeisser--Triebel \cite{ST1987}, Triebel \cite{T1983} and Xiong--Xu--Yin \cite{XXY2018}.
 }
In the following, $\T_\lambda=\R/(2\pi\lambda\Z)$ denotes the one-dimensional torus whose period is $2\pi\lambda$ for $\lambda>0$.\par
We define $\sD(\T_\lambda)$ and its subspace $\sD_0(\T_\lambda)$ as 
\begin{align*}
	& \sD(\T_\lambda)
	= \left\{
		\phi\in C^\infty (\R) 
		\mid \text{ $ \phi(z)=\phi(z') $ for $z-z'\in 2\pi\lambda\Z$ }
		\right\}, \\
	& \sD_0(\T_\lambda)
	= \left\{
		\phi\in \sD(\T_\lambda)
		\mid \int_{[-\pi\lambda, \pi\lambda]} \phi(z) \, dz = 0
		\right\}.
\end{align*}
Here the locally convex topology in $\sD(\T_\lambda)$ is generated by the semi-norms 
\[
	[ \phi ]_{l} = \sup_{z\in \T_\lambda} | \partial_z^{l}\phi (z) |
	\qquad \text{for $l\in\N\cup\{0\}$},
\]
and we equip $\sD_0(\T_\lambda)$ with the relative topology of $\sD(\T_\lambda)$.
We define $\sD^\prime (\T_\lambda)$ (resp.~$\sD_0^\prime (\T_\lambda)$) as the topological dual space of $\sD(\T_\lambda)$ (resp.~$\sD_0(\T_\lambda)$), respectively. \par
Any $\phi\in\sD(\T_\lambda)$ has the Fourier series representation
\[
	\phi(z) = \sum_{k\in\Z} \,
		[\sF_{\T_\lambda} \phi](k) e^{i k\cdot (\lambda^{-1} z)}
	\qquad \text{in $\sD(\T_\lambda)$},
\]
where $\sF_{\T_\lambda}\phi$ is defined by
\[
	[\sF_{\T_\lambda} \phi](k)
	\equiv \frac{1}{2\pi\lambda}
		\int_{[-\pi\lambda, \pi\lambda]} \phi(z) e^{-ik\cdot (\lambda^{-1} z)} \, dz
	\qquad \text{for $k\in\Z$}.
\]
We can easily see that the Fourier coefficients $[\sF_{\T_\lambda} \phi] (k)$ are rapidly decreasing as $|k|\to\infty$, namely, $\sF_{\T_\lambda} \phi \in \sS_\lambda(\Z)$, where
{\small
\[
	\sS_\lambda(\Z)
	\equiv \left\{ \{a_k\}_{k\in\Z}\subset \C
		\mid \text{for any $m\in\N\cup\{0\}$, there exists $C_m>0$ s.t.~
		$\disp |a_k| \leq \frac{C_m}{(1+|k/\lambda|)^m} $ }\right\}.
\]
}
We equip $\sS_\lambda(\Z)$ with the locally convex topology generated by the norms
\[
	\| \{a_k\}_{k\in\Z} \|_{\ell^\infty_{m, \lambda}}
	\equiv \sup_{k\in\Z} \left( (1+|k/\lambda|)^m|a_k| \right),
	\qquad \text{for $m\in\N\cup\{0\}$}.
\]
We also define $\sF_{\T_{\lambda}}^{-1}: \sS_\lambda(\Z)\to \sD(\T_\lambda)$ as
\[
	[\sF^{-1}_{\T_\lambda}g](z)
	\equiv \sum_{k\in\Z} g(k) e^{ik\cdot(\lambda^{-1}z)}.
\]
Here we also define $\sF_{\T_\lambda}: \sD^\prime(\T_\lambda)\to \sS^\prime_\lambda(\Z)$ and $\sF_{\T_{\lambda}}^{-1}: \sS^\prime_\lambda(\Z)\to \sD^\prime(\T_\lambda)$ as 
\[
	\langle{ \sF_{\T_{\lambda}}f , \alpha }\rangle_{\sS_\lambda^\prime(\Z)\times \sS_\lambda(\Z)}
	\equiv \langle{ f, \sF^{-1}_{\T_\lambda}[\alpha(-\cdot )] 
		}\rangle_{\sD^\prime(\T_\lambda)\times\sD(\T_\lambda)}
	\qquad \text{for any $\alpha\in\sS_\lambda(\Z)$},
\]
and
\[
	\langle{ \sF^{-1}_{\T_{\lambda}}g , \phi }\rangle_{\sD^\prime(\T_\lambda)\times\sD(\T_\lambda)}
	\equiv \langle{ g, \sF_{\T_\lambda}[\phi(-\cdot)]
		}\rangle_{\sS_\lambda^\prime(\Z)\times \sS_\lambda(\Z)}
	\qquad \text{for any $\phi\in\sD(\T_\lambda)$},
\]
respectively, where $\sS^\prime_\lambda(\Z)$ is the topological dual space of $\sS_\lambda(\Z)$.\par
We also define the convolution of $\phi, \psi\in \sD(\T_\lambda)$ (or $\phi, \psi\in \sD_0(\T_\lambda)$) as
\[
	(\phi\ast \psi)(z) = \int_{[-\pi\lambda, \pi\lambda]} \phi(z-\zeta) \psi(\zeta)\,d\zeta.
\]
Using this definition, we define the convolution of $\phi\in \sD(\T_\lambda)$ and $f\in \sD^\prime(\T_\lambda)$ (resp.~$\phi\in \sD_0(\T_\lambda)$ and $f\in \sD_0^\prime(\T_\lambda)$) as
\[
	\langle{ \phi \ast f, \psi }\rangle
	\equiv \langle{ f, \phi(-\cdot)\ast \psi }\rangle
	= \left\langle{ f, \int_{[-\pi\lambda, \pi\lambda]} \phi(y-\cdot)\psi(y)\,dy }\right\rangle 
\]
for any $\psi\in\sD(\T_\lambda)$ (resp.~$\psi\in\sD_0(\T_\lambda)$).
\par
Next, we recall some properties of Littlewood--Paley decomposition. We first take a smooth and radial symmetry function $\chi\in C_0^\infty(\R)$ satisfying
\[
\chi(\xi) = \chi (|\xi|) =\begin{cases}
1 & (|\xi|\leq1/2), \\
0 & (|\xi|\geq1).
\end{cases}
\]
We define the function $\rho$ on $\R$ as $\rho(\xi) \equiv \chi(\xi/2)-\chi(\xi)$. Then, it is easy to see that $\{\rho_{\lambda, j}\}_{j\in\Z}$ defined as
\begin{equation}\label{eq: rhoj}
	\rho_{\lambda, j}(\xi) \equiv \rho(2^{-j}\lambda^{-1}\xi)
\end{equation}
satisfies the following properties.

\begin{prop}\label{prop: LP}
Let $\lambda>0$, and let $\{\rho_{\lambda, j}\}_{j\in\Z}$ be the partition of the unity defined as \eqref{eq: rhoj}.
\begin{enumerate}[(1)]
\item $\sum\limits_{j\in\Z}\rho_{\lambda, j}(\xi)=1 \qquad \text{for $\xi\neq0$}$.
\item $\supp{ \rho_{\lambda, j} } 
		\subset \left\{ \xi\in\R \mid 2^{j-1}\leq |\lambda^{-1}\xi|\leq 2^{j+1} \right\}$
\item $\supp{ \rho_{\lambda, j} }\cap \supp{ \rho_{\lambda, j'} }
		=\emptyset \qquad \text{for $|j-j'|\geq2$}.$
\end{enumerate}	
\end{prop}

Let $\vp_{\lambda, j} = \sF_{\T_\lambda}^{-1}[\rho_{\lambda, j}|_{\Z}]$, where $\rho_{\lambda, j}|_{\Z}$ denotes the restriction of $\rho_{\lambda, j}$ on $\Z$. From \autoref{prop: LP}, it follows $\vp_{\lambda, j}\in \sD(\T_\lambda)$.

Next we give the definition of the homogeneous toroidal Besov spaces on $\T_\lambda$.

\begin{defn}[homogeneous toroidal Besov spaces, cf.~\cite{T2019}]
Let $s\in \R$, $1\leq p \leq \infty$ and $1\leq q \leq \infty$. We define the homogeneous toroidal Besov space $\dBrpq{s}{p}{q}(\T_\lambda)$ as the set of all distribution $u\in \sD_0^\prime(\T_\lambda)$ satisfying 
\[
	\| u \|_{\dBrpq{s}{p}{q}(\T_\lambda)} = \begin{cases} \vspace{10pt}
	\disp \left( \sum_{j\in\Z} 
		\left| 2^{sj} \| \vp_{\lambda, j} \ast u \|_{L^p(\T_\lambda)} \right|^q \right)^{1/q} < \infty 
		& (1\leq q<\infty), \\
	\disp \sup_{j\in\Z} \left( 2^{sj} \| \vp_{\lambda, j} \ast u \|_{L^p(\T_\lambda)} \right) < \infty & (q=\infty). 
	\end{cases}
\]
\end{defn}

\begin{rem}
For the definition of inhomogeneous Besov spaces on a torus, readers can refer \cite{ST1987, T1983, XXY2018}.
\end{rem}

By the definition of the norm of Besov spaces, it is to see that the following property holds.
\begin{lem}
Let $s\in\R$, $1\leq p\leq \infty$ and $1\leq q_1 \leq q_2 \leq \infty$. Then, the following inclusion relation holds:
\[
	\dBrpq{s}{p}{q_1}(\T_\lambda) \subset \dBrpq{s}{p}{q_2}(\T_\lambda).
\]
\end{lem}

For $p=2$, the following remarkable properties hold. 

\begin{prop}\label{prop:equiv}
\begin{enumerate}[(1)]
\item Let $l=\N\cup\{0\}$.
	Then, there exist absolutely constants $C_1=C_1(l, \lambda)$ and $C_2=C_2(l, \lambda)$ such that
	\begin{equation}\label{eq: equiv}
	C_1\| u \|_{\dot{W}^{l, 2}(\T_\lambda)} \leq  \| u \|_{\dBrpq{l}{2}{2}(\T_\lambda)}
	\leq C_2 \| u \|_{\dot{W}^{l, 2}(\T_\lambda)} 
	\end{equation}
	for any $u\in \dot{W}^{l, 2}(\T_\lambda)$, where 
	\begin{align*}
	&\, \dot{W}^{l, 2}(\T_\lambda) \\
	=&\,
		\begin{cases}
		\vspace{10pt}
		\left\{ u\in L_{\loc}^1(\T_\lambda) \mid \substack{ \disp
			\| u \|_{\dot{W}^{l, 2}(\T_\lambda)} \equiv \| \dz^l u \|_{L^2(\T_\lambda)}<\infty, where \\  
			\text{{\normalsize $\dz^l$ denotes the $l$-th weak derivative. }}
			} \right\} & (l\geq1), \\
		\disp
		\dot{L}^2(\T_\lambda)
			= \left\{ u\in L^2(\T_\lambda) \mid 
				\| u \|_{\dot{L}^2(\T_\lambda)}
				\equiv \left\| u - [\sF_{\T_\lambda}u](0) \right\|_{L^2(\T_\lambda)}<\infty \right\}
			& (l=0).
		\end{cases} 
	\end{align*}

\item In particular, it holds that
	\begin{equation}\label{eq: equiv L^2}
	\| u \|_{\dBrpq{0}{2}{2}(\T_\lambda)} \leq C \| u \|_{L^2(\T_\lambda)}
	\end{equation}
	for any $u \in L^2(\T_\lambda)$, where $C=C(\lambda)>0$.
\end{enumerate}
\end{prop}

\begin{rem}
In the homogeneous Besov spaces on the whole space $\R$ (or more generally, $\R^n$), it is well known that the same property as \autoref{prop:equiv} (1) are valid. For example, Bahouri--Chemin--Danchin \cite[p.63]{BCD2011} mentions about this for the Riesz potential spaces $\dot{H}^s$.
\end{rem}

We also list the important properties of the toroidal Besov spaces.
As in \autoref{prop:equiv}, the following proposition is also valid for the whole space case.

\begin{prop}\label{prop:Sobolev}
\begin{enumerate}[(1)]
\item We define the one-dimensional fractional derivative operator $(-\p_{zz})^{s/2}$ ($s\in\R$) on $\T_\lambda$ as
	\[
		(-\p_{zz})^{s/2}f 
		\equiv \sF_{\T_\lambda}^{-1}[ |\lambda^{-1}\cdot|^{s} \sF_{\T_\lambda}f ]
		\qquad \text{for $f\in\sD_0^\prime(\T_\lambda)$}.
	\]
	For $s\in\R$ and $1\leq p\leq \infty$, Bernstein's inequality holds in the sense that
	\[
		\| (-\partial_{zz}^2)^{s/2} \vp_{\lambda, j} \ast u \|_{L^p(\T_\lambda)}
		\leq C 2^{sj} \| \vp_{\lambda, j} \ast u \|_{L^p(\T_\lambda)},
	\]
	where $C=C(s)>0$ does not depend on $p$ and $\lambda$.
	
\item Let $1\leq p \leq q \leq \infty$. It holds that
	\[
		 \| \vp_{\lambda, j} \ast u \|_{L^q}
		\leq C 2^{(1/p-1/q)j} \| \vp_{\lambda, j} \ast u \|_{L^p},
	\]
	where $C=C(p, q, \lambda)>0$.
	
\item Let $-\infty<s_0\leq s_1<\infty$, $1\leq p_1\leq p_0\leq \infty$ and $1\leq q \leq \infty$ satisfy $s_1-1/p_1 = s_0-1/p_0$.
	Then, the Sobolev-type continuous embedding
	\[
		\dBrpq{s_1}{p_1}{q} \hookrightarrow \dBrpq{s_0}{p_0}{q} 
	\]
	holds with the estimate
	\begin{equation}\label{eq: Sobolev}
		\| u \|_{\dBrpq{s_0}{p_0}{q}} 
		\leq C \| u \|_{\dBrpq{s_1}{p_1}{q}},
	\end{equation}
	where $C=C(p_0, p_1, q, s_0, s_1, \lambda)>0$.
\end{enumerate}
\end{prop}

\subsection{Weak solutions to the primitive equations}
\label{subsec:weak sol}

\begin{defn}[weak solution to the 3D primitive equations]\label{defn:weak solution PE} 
	A function $v: \T^3\times (0, \infty)\to \R^{  3  }$ is called a \emph{weak solution} to the primitive equations \eqref{eq:PE} if the following conditions hold: 
	
	\begin{enumerate}[(i)]
		\item $v\in C_w([0, \infty);   \mathcal{H}  ) \cap L^2_\loc([0, \infty); (H^1\cap   \mathcal{H}  )(\T^3))$;
		\item $v$ satisfies \eqref{eq:PE} in the weak sense:
		\[
		-
		\int_0^\infty \int_{\T^3} v \cdot \p_\tau \phi \,dx d\tau 
		+ \int_0^\infty \int_{\T^3} \nabla v \, \colon \nabla \phi \,dx d\tau
		- 
		\int_0^\infty \int_{\T^3} (u \otimes v) \, \colon \nabla \phi \,dx d\tau
		= (v_0, \phi(0))
		\]
		for any compactly supported in time function 
			$\phi \in C([0, \infty); (C^1\cap \mathcal{H})(\T^3))\cap C^1([0, \infty); C(\T^3))$.
	\end{enumerate}
	Here, $A\,\colon B$ denotes the inner-product of   $3 \times 2$-  matrixes $A$ and $B$.
\end{defn}

\begin{defn}[weak solution to the 2D primitive equations]\label{defn:weak solution PE2}
	A function $v: \T^2\times (0, \infty)\to \R^3$ is called a \emph{weak solution} to the two-dimensional primitive equations \eqref{eq:PE2} if the following conditions hold: 
	
	\begin{enumerate}[(i)]
		\item $v={}^\top(v^1, v^2)\in C_w([0, \infty); \mathcal{H}) \cap L^2_\loc([0, \infty); (H^1\cap \mathcal{H})$;
		\item $v$ satisfies \eqref{eq:PE2} in the weak sense:
		\begin{align*}
		& -
		\int_0^\infty \int_{\T^2} v \cdot \p_\tau \phi \,dx d\tau 
		+ \int_0^\infty \int_{\T^2} \nabla_{x_1, z} v \, \colon\nabla_{x_1, z} \phi \,dx d\tau \\
		& \hspace{130pt}
		- 
		\int_0^\infty \int_{\T^2} ({}^\top(v^1, w)\otimes v) \, \colon \nabla_{x_1, z} \phi \,dx d\tau
		= (v_0, \phi(0))
		\end{align*}
		for any compactly supported in time function 
			$\phi \in C([0, \infty); (C^1\cap \mathcal{H})(\T^2))\cap C^1([0, \infty); C(\T^2))$.
	\end{enumerate}
	In this two-dimensional case, $A\,\colon B$ denotes the inner-product of $2\times 2$-matrixes $A$ and $B$.
\end{defn}

The following proposition states that any weak solutions on $[0, \infty)$ are also weak solutions on $[s, t]$ ($0\leq s<t<\infty$). For the proof, see \cite[Proposition 3.2]{LT2017}. The two-dimensional case may be proved in the same way.
\begin{prop}[cf.~{\cite[Proposition 3.2]{LT2017}}]
\begin{enumerate}[(1)]
\item Let $v$ be a weak solution to the three-dimensional primitive equations \eqref{eq:PE}. Then, for any $0\leq s < t<\infty$, it holds that
	\begin{equation}\label{eq:weak solution PE}
		\left. \int_{\T^3} v\cdot\phi\,dx \right|_{\tau=s}^{\tau=t} 
			- \int_s^t \int_{\T^3} v\cdot \p_\tau\phi\,dxd\tau 
			+ \int_s^t \int_{\T^3} \nabla v : \nabla\phi \, dxd\tau
			- \int_s^t \int_{\T^3} (u\otimes v): \nabla\phi \,dxd\tau = 0
	\end{equation}
	for any $\phi \in C([s, t]; (C^1\cap \mathcal{H})(\T^3))\cap C^1([s, t]; C(\T^3))$.
\item Let $v$ be a weak solution to the two-dimensional primitive equations \eqref{eq:PE2}. Then, for any $0\leq s < t<\infty$, it holds that
	\begin{align}
	\begin{aligned}\label{eq:weak solution PE2}
		& \left. \int_{\T^2} v\cdot\phi\,dx \right|_{\tau=s}^{\tau=t} 
			- \int_s^t \int_{\T^2} v\cdot \p_\tau\phi \,dxd\tau 
			+ \int_s^t \int_{\T^2} \nabla_{x_1, z} v : \nabla_{x_1, z}\phi \, dxd\tau \\
		&\hspace{150pt}
			- \int_s^t \int_{\T^2} ( {}^\top (v^1, w)\otimes v): \nabla_{x_1, z}\phi \,dxd\tau = 0
	\end{aligned}
	\end{align}
	for any $\phi \in C([s, t]; (C^1\cap \mathcal{H})(\T^2))\cap C^1([s, t]; C(\T^2))$.
\end{enumerate}
\end{prop}

In \autoref{thm:unique}, we imposed the energy inequality \eqref{eq:energy inequality vj} to the weak solutions $v_1$ and $v_2$. This is because the following remarkable lemma can be deduced. For the detailed proof, see \cite{J2021, LT2017}.
\begin{lem}[cf.~{\cite[Theorem 3.1, Theorem 3.3]{J2021}}, {\cite[Proposition 3.1, Proposition 3.4]{LT2017}}]\label{lem:prop weak sol}
Let $v$ be a weak solution to the three-dimensional primitive equations \eqref{eq:PE} or the two-dimensional one \eqref{eq:PE2}. Suppose that 
	\begin{equation}\label{eq:energy inequality}
	\frac{1}{2}\| v(t) \|_{L^2(\T^n)}^2 + \int_0^t \| \nabla v (\tau) \|_{L^2(\T^n)}^2 \,d\tau
	\leq \frac{1}{2}\| v_0 \|_{L^2(\T^n)}^2 \qquad \text{($n=3$ or $2$)}
	\end{equation}
	for a.e.~$t\in (0, \infty)$. \par
Then, the followings hold.
\begin{enumerate}
\item There exists a subset $E\subset (0, \infty)$ whose measure is zero such that 
	\[
	\lim_{\substack{t\to 0+,\\ t\in (0, \infty)\setminus E}} v(t)=  v_0 \qquad \text{in $L^2(\T^n)$}.
	\]
\item If $v_0\in H^1\cap \mathcal{H}$, then $v$ coincides with the unique global strong solution with the same initial data $v_0$.
\item $v$ is right continuous in $L^2(\T^n)$ at $t=0$, i.e., $v(t)\to v_0$ as $t\to +0$ in $L^2(\T^n)$. 
\end{enumerate}
\end{lem}

\begin{rem}\label{rem:condition (iii)}
	\begin{enumerate}[(1)]
	\item Note that (3) implies additional time regularity $v~\in~BC([0, \infty); L^2(\T^n))$.
	\item It should be noted that Li--Titi \cite{LT2017} use a stronger definition of weak solution (see \cite[Definition 1.1]{LT2017}) which requires the energy inequality and an energy-type differential inequality. In particular, this differential inequality is not used in the proof.
	\end{enumerate}
\end{rem}

\section{Proof of the uniqueness}
\label{sec:proof}

\subsection{3D case}\label{subsec:uniqueness 3D}

In this \autoref{subsec:uniqueness 3D}, let $v_d \equiv v_1 - v_2$.
In order to prove the uniqueness, we introduce the following lemma which is a consequence of \cite[Corollary 3.1 (i)]{LT2017}.
\begin{lem}\label{lem:1st form}
Let $v_1$ and $v_2$ be weak solutions to the three-dimensional primitive equations \eqref{eq:PE} with the energy inequality.
\par
Then, we have
\begin{align}
\begin{aligned}\label{eq:1st form}
	&\, \frac{1}{2}\| v_d (t) \|_{L^2}^2 - \frac{1}{2}\| v_d(s) \|_{L^2}^2
	+ \int_s^t \| \nabla v_d (\tau) \|^2_{L^2} \,d \tau \\
	=&\, \int_s^t \int_{\T^3} ( u_d \otimes v_1 )(x, \tau) \colon \nabla v_d (x, \tau)  \,d x d \tau \\
	=&\, \int_s^t \int_{\T^3} ( v_d \otimes v_1 )(x, \tau) \colon \dH v_d (x, \tau)  \,d x d \tau 
		+ \int_s^t \int_{\T^3}  w_d(x, \tau) v_1(x, \tau) \cdot \dz v_d (x, \tau)  \,d x d \tau
\end{aligned}
\end{align}
for $0<s\leq t\leq T$.
\end{lem}

\begin{proof}
From \cite[Corollary 3.1 (i)]{LT2017}, it follows that every weak solution with the energy inequality is smooth. In particular, $v_1$ and $v_2$ are smooth. Therefore, we obtain \eqref{eq:1st form} by a direct calculation.
\end{proof}

With the aid of the Littlewood--Paley decomposition and the Bony's decomposition, we can decompose the right hand side of \eqref{eq:1st form} into the following six terms:
\begin{align}
\begin{aligned}\label{eq:RHS 1st form}
\text{R.H.S.~ of \eqref{eq:1st form}} \ 
	&= \sum_{j\in\Z} \sum_{k=-3}^3 \int_s^t \int_{\T^3}
	\left[ \vpt_j \ast \left( ( \vp_{j+k} \ast v_d ) \otimes ( P_{j+k}v_1 ) \right) \right]
		\colon \left[ \vp_j \ast ( \dH v_d ) \right] \,dx d\tau \\ 
	&\quad + \sum_{j\in\Z} \sum_{k=-3}^3 \int_s^t \int_{\T^3}
	\left[ \vpt_j \ast \left( ( P_{j+k} v_d ) \otimes ( \vp_{j+k}\ast v_1 ) \right) \right]
		\colon \left[ \vp_j \ast ( \dH v_d ) \right] \,dx d\tau \\ 
	&\quad +\sum_{j\in\Z} \sum_{ \substack{ \max{ \{ k, l \} } \\ \qquad  \geq  j-3 }  } 
		\sum_{ |k-l| \leq 2 } \int_s^t \int_{\T^3} 
	\left[ \vpt_j \ast \left( ( \vp_{k} \ast v_d ) \otimes ( \vp_{l} \ast v_1 ) \right) \right]
		\colon \left[ \vp_j \ast ( \dH v_d ) \right] \,dx d\tau \\	
	&\quad + \sum_{j\in\Z} \sum_{k=-3}^3 \int_s^t \int_{\T^3}
	\left[ \vpt_j \ast \left( ( \vp_{j+k} \ast w_d ) ( P_{j+k}v_1 ) \right) \right]
		\cdot \left[ \vp_j \ast ( \dz v_d ) \right] \,dx d\tau \\ 
	&\quad + \sum_{j\in\Z} \sum_{k=-3}^3 \int_s^t \int_{\T^3}
	\left[ \vpt_j \ast \left( ( P_{j+k}w_d ) ( \vp_{j+k} \ast v_1 ) \right) \right]
		\cdot \left[ \vp_j \ast ( \dz v_d ) \right] \,dx d\tau \\ 
	&\quad + \sum_{j\in\Z} \sum_{ \substack{ \max{ \{ k, l \} } \\ \qquad  \geq  j-3 }  } 
		\sum_{ |k-l| \leq 2 } \int_s^t \int_{\T^3}
	\left[ \vpt_j \ast \left( ( \vp_{k} \ast w_d ) ( \vp_{l} \ast v_1 ) \right) \right]
		\colon \left[ \vp_j \ast ( \dz v_d ) \right] \,dx d\tau \\ 
	&\eqqcolon J_1 + J_2 + J_3 + J_4 + J_5 + J_6,
\end{aligned}
\end{align}
where $\ast$ denotes the one-dimensional convolution on $\T$ in $z$-direction, and where $\disp P_{m} f \equiv \sum_{j\leq m-3} \vp_j \ast f$.

In the following, we will show 
\begin{equation}\label{eq: desired}
	J_1 + \cdots + J_6
	\leq \int_s^t \| \nabla v_d (\tau) \|_{L^2}^2	
	+ C \int_s^t \| v_d (\tau) \|_{L^2}^2 
		\left( \| v_1 (\tau) \|_{ \dBrpqz{-1/p}{p}{\infty}\LH{\infty} }^\beta 
		+ \| v_1 \|_{\dBrpq{2/\gamma}{2}{2}\LH{\infty}}^{\gamma} \right) \, d\tau,
\end{equation}
where the absolute constant $C$ does not depend on $s$ and $t$, so that we obtain
\begin{align}
\begin{aligned}\label{eq: desired 2}
	\frac{1}{2}\| v_d (t) \|_{L^2}^2 - \frac{1}{2}\| v_d(s) \|_{L^2}^2
	&\leq 
	C \int_s^t \| v_d (\tau) \|_{L^2}^2 
	\left( \| v_1 (\tau) \|_{ \dBrpqz{-1/p}{p}{\infty}\LH{\infty} }^\beta 
		+ \| v_1 \|_{\dBrpq{2/\gamma}{2}{2}\LH{\infty}}^{\gamma} \right) \, d\tau \\
	&\leq C\int_0^t \| v_d (\tau) \|_{L^2}^2 
	\left( \| v_1 (\tau) \|_{ \dBrpqz{-1/p}{p}{\infty}\LH{\infty} }^\beta 
		+ \| v_1 \|_{\dBrpq{2/\gamma}{2}{2}\LH{\infty}}^{\gamma} \right) \, d\tau
\end{aligned}
\end{align}
for $0 < s \leq t \leq T$. If \eqref{eq: desired} and \eqref{eq: desired 2} are proved, then we have
\begin{align*}
&\, \frac{1}{2}\| v_d (t) \|_{L^2}^2 - \frac{1}{2}\| v_d(0) \|_{L^2}^2 \\
\leq&\, C \int_0^t \| v_d (\tau) \|_{L^2}^2 
	\left( \| v_1 (\tau) \|_{ \dBrpqz{-1/p}{p}{\infty}\LH{\infty} }^\beta 
	+ \| v_1 (\tau) \|_{ \dBrpqz{2/\gamma}{2}{2}\LH{\infty} }^\gamma \right) \, d\tau
	+ \frac{1}{2} \left( \| v_d (s) \|_{L^2}^2 - \| v_d(0) \|_{L^2}^2 \right).
\end{align*}
We can deduce from \autoref{lem:prop weak sol} (3) that
\[
\frac{1}{2}\| v_d (t) \|_{L^2}^2 - \frac{1}{2}\| v_d(0) \|_{L^2}^2
\leq C \int_0^t \| v_d (\tau) \|_{L^2}^2 
	\left( \| v_1 (\tau) \|_{ \dBrpqz{-1/p}{p}{\infty}\LH{\infty} }^\beta 
		+ \| v_1 \|_{\dBrpq{2/\gamma}{2}{2}\LH{\infty}}^{\gamma} \right) \, d\tau.
\]
Applying Gronwall's inequality, we have by the assumption $v_1\in L^\beta(0, T; \dot{B}^{-1/p}_{p, \infty, z}\LH{\infty})\cap L^\gamma(0, T; \dot{B}^{2/\gamma}_{2, 2, z}\LH{\infty})$ and $v_d(0)=v_1(0)-v_2(0)=v_0-v_0=0$ that
\[
\| v_d (t) \|_{L^2}^2 \leq C \exp{ \int_0^t \left( \| v_1 (\tau) \|_{ \dBrpqz{-1/p}{p}{\infty}\LH{\infty} }^\beta
	 + \| v_1 \|_{\dBrpq{2/\gamma}{2}{2}\LH{\infty}}^{\gamma}
	  \right) \, d\tau } \cdot \| v_d(0) \|_{L^2}^2 = 0.
\]
Therefore, we have $v_d\equiv 0$ on $[0, T]$. \par

\bigskip
\noindent{\it Proof of \eqref{eq: desired}.} {\it \underline{Step 1.} Estimate of $J_1$.} \par
We can estimate $J_1$ by H\"{o}lder's inequality as follows:
\begin{align}
\begin{aligned}\label{eq: J1_1}
J_1 &= 
	\sum_{j\in\Z} \sum_{k=-3}^3 \int_s^t \int_{\T^3}
	\left[ \vpt_j \ast \left( ( \vp_{j+k} \ast v_d ) \otimes ( P_{j+k}v_1 ) \right) \right]
		\colon \left[ \vp_j \ast ( \dH v_d ) \right] \,dx d\tau \\	
	& \leq \sum_{j\in\Z} \sum_{k=-3}^3 \int_s^t \int_{\T^2}
		\| \vpt_j \ast \left( ( \vp_{j+k} \ast v_d ) \otimes ( P_{j+k}v_1 ) \right) \|_{L^2_z}
		\| \vp_j \ast ( \dH v_d ) \|_{L^2_z} \, d \xH d\tau \\
	& \leq C \sum_{j\in\Z} \sum_{k=-3}^3 \int_s^t \int_{\T^2}
		\| \vp_{j+k} \ast v_d \|_{L^q_z}  \| P_{j+k}v_1 \|_{L^{p}_z}
		\cdot \| \vp_j \ast ( \dH v_d )  \|_{L^2_z} \, d \xH d\tau \\
	&\leq C \sum_{j\in\Z} \sum_{k=-3}^3 \int_s^t \int_{\T^2}
		 \| \vp_{j+k} \ast v_d \|_{L^q_z}
		 \left[ \sum_{l \leq j+k-3}  \| \vp_l \ast v_1 \|_{L^{p}_z} \right]
		\cdot \| \vp_j \ast ( \dH v_d) \|_{L^2_z} \, d \xH d\tau \\
	&\leq C \sum_{j\in\Z} \sum_{k=-3}^3 \int_s^t 
		\left\| \| \vp_{j+k} \ast v_d \|_{L^q_z} \right\|_{\LH{2}}
		\left\| \sum_{l \leq j+k-3}  { \| \vp_l \ast v_1 \|_{L^p_z} } \right\|_{ \LH{\infty} }
		\cdot  \| \vp_j \ast (\dH v_d ) \|_{L^2} \, d\tau
\end{aligned}
\end{align}
where $C=C(n, p)>0$.
Interpolation inequality, \autoref{prop:equiv} (1) for $l=1$, and \autoref{prop:Sobolev} (2) yield that, for any $j' \in \Z$,
\begin{align*}
	\| \vp_{j'} \ast v_d \|_{L^q_z} 
	&\leq \| \vp_{j'} \ast v_d \|_{L^2_z}^{1-2/p} 
		\| \vp_{j'} \ast v_d \|_{L^\infty_z}^{2/p} \\
	&\leq C \| \vp_{j'} \ast v_d \|_{L^2_z}^{1-2/p}
		 \left[ 2^{j'/2} \| \vp_{j'} \ast v_d \|_{L^2_z} \right]^{2/p} \\
	&\leq C 2^{-j'/p} \| \vp_{j'} \ast v_d \|_{L^2_z}^{1-2/p}
		 \left[ 2^{j'} \| \vp_{j'} \ast  v_d \|_{L^2_z} \right]^{2/p}  \\
	&\leq C 2^{-j'/p} \| \vp_{j'} \ast v_d \|_{L^2_z}^{1-2/p}
		\| \vp_{j'} \ast ( \dz v_d ) \|_{L^2_z}^{2/p}.
\end{align*}
Therefore, it follows from H\"{o}lder's inequality that
\begin{align}
\begin{aligned}\label{eq: Lq-interpolation}
	\left\| \| \vp_{j'} \ast v_d \|_{L^q_z} \right\|_{ \LH{2} }
	&\leq C2^{-j'/p} \left\| \| \vp_{j'} \ast v_d \|_{L^2_z}^{1-2/p}
		\| \vp_{j'} \ast ( \dz v_d) \|_{L^2_z}^{2/p} \right\|_{ \LH{2} } \\
	&\leq C2^{-{j'}/p} \| \vp_{j'} \ast v_d \|_{L^2}^{1-2/p}
		\| \vp_{j'} \ast ( \dz v_d) \|_{L^2}^{2/p} .
\end{aligned}
\end{align}

Minkowski's inequality and the definition of the toroidal Besov space yield that
\begin{align}
\begin{aligned}\label{eq: v1-lowfreq}
	\left\| \sum_{l \leq j+k-3}  { \| \vp_l \ast v_1 \|_{L^p_z} } \right\|_{ \LH{\infty} }
	&\leq \sum_{l \leq j+k-3} \| \vp_l \ast v_1 \|_{ L^p_z\LH{\infty} } \\
	&\leq C 2^{(j+k-3)/p} \|v_1\|_{ \dBrpqz{-1/p}{p}{\infty}\LH{\infty} } \\
	&\leq C 2^{(j+k)/p} \|v_1\|_{ \dBrpqz{-1/p}{p}{\infty}\LH{\infty} }. 
\end{aligned}
\end{align}

Applying the estimates \eqref{eq: Lq-interpolation} and \eqref{eq: v1-lowfreq} for \eqref{eq: J1_1}, we have that

\begin{align*}
	J_1 
	&\leq C \sum_{j\in\Z} \sum_{k=-3}^3 \int_s^t 
		2^{-(j+k)/p} \| \vp_{j+k} \ast v_d \|_{L^2}^{1-2/p} \| \vp_{j+k} \ast ( \dz v_d) \|_{L^2}^{2/p} \\
		& \hspace{150pt}
		\cdot 2^{(j+k)/p} \|v_1\|_{ \dBrpqz{-1/p}{p}{\infty}\LH{\infty} } 
		 \cdot \| \vp_j \ast (\dH v_d ) \|_{L^2}\,  d\tau \\
	&\leq C \sum_{j\in\Z} \sum_{k=-3}^3 \int_s^t 
		\| \vp_j \ast ( \dH{v_d} ) \|_{L^2} 
		\cdot \| \vp_{j+k} \ast (\dz v_d) \|_{L^2}^{2/p}
		\cdot \| \vp_{j+k} \ast v_d \|_{L^2}^{1-2/p}
		\cdot \|v_1\|_{ \dBrpqz{-1/p}{p}{\infty}\LH{\infty} } \, d\tau.
\end{align*}
Young's inequality and \autoref{prop:equiv} (2) yield that
\begin{align}
\begin{aligned}\label{eq: J1}
J_1 &\leq \sum_{j\in\Z} \sum_{k=-3}^3  \int_s^t 
	\frac{1}{42c^2} \| \vp_j \ast ( \dH{v_d} ) \|_{L^2}^2 
	+ \frac{1}{42c^2} \| \vp_{j+k} \ast (\dz v_d) \|_{L^2}^2 \, d\tau \\
	& \hspace{150pt} 
	+ C \sum_{j\in\Z} \sum_{k=-3}^3 \int_s^t 
	\| \vp_{j+k} \ast v_d \|_{L^2}^2 \|v_1\|_{ \dBrpqz{-1/p}{p}{\infty}\LH{\infty} }^\beta \, d\tau \\
	&\leq \frac{1}{6c^2} \sum_{j\in\Z} \int_s^t
		\left( \| \vp_j \ast ( \dH{v_d} ) \|_{L^2}^2 
			+  \| \vp_j \ast (\dz v_d) \|_{L^2}^2 \right) \, d\tau \\
	&\hspace{150pt}
	+ C \sum_{j\in\Z} \int_s^t 
		\| \vp_j \ast v_d \|_{L^2}^2 \|v_1\|_{ \dBrpqz{-1/p}{p}{\infty}\LH{\infty} }^\beta \, d\tau \\
	&= \frac{1}{6c^2} \sum_{j\in\Z} \int_s^t
		\| \vp_j \ast ( \nabla{v_d} ) \|_{L^2}^2 \, d\tau
	+ C \sum_{j\in\Z} \int_s^t 
		\| \vp_j \ast v_d \|_{L^2}^2 \|v_1\|_{ \dBrpqz{-1/p}{p}{\infty}\LH{\infty} }^\beta \, d\tau \\
	&= \frac{1}{6c^2}
		 \int_s^t \left\| \| v_d \|_{ \dBrpqH{0}{2}{2} } \right\|_{L_z^2}^2 \, d\tau
	+ C \int_s^t 
		\left\| \| v_d \|_{ \dBrpqH{0}{2}{2} } \right\|_{L_z^2}^2
		 \|v_1\|_{ \dBrpqz{-1/p}{p}{\infty}\LH{\infty} }^\beta \, d\tau \\
	&\leq \frac{1}{6} \int_s^t \| v_d \|_{L^2}^2 \, d\tau
	+ C \int_s^t \| v_d \|_{L^2}^2
		 \|v_1\|_{ \dBrpqz{-1/p}{p}{\infty}\LH{\infty} }^\beta \, d\tau
\end{aligned}
\end{align}
where the absolute constant $c$ is given by \eqref{eq: equiv L^2}. \par

\medskip
\noindent{\it \underline{Step 2.} Estimate of $J_2$.} 

\begin{align}
\begin{aligned}\label{eq: J2_1}
J_2 &= 
	\sum_{j\in\Z} \sum_{k=-3}^3 \int_s^t \int_{\T^3}
	\left[ \vpt_j \ast \left( ( P_{j+k} v_d ) \otimes ( \vp_{j+k}\ast v_1 ) \right) \right]
		\colon \left[ \vp_j \ast ( \dH v_d ) \right] \,dx d\tau \\	
	& \leq \sum_{j\in\Z} \sum_{k=-3}^3 \int_s^t \int_{\T^2}
		\| \vpt_j \ast \left( ( P_{j+k} v_d ) \otimes ( \vp_{j+k}\ast v_1 ) \right) \|_{L^2_z}
		\| \vp_j \ast ( \dH v_d ) \|_{L^2_z} \, d \xH d\tau \\
	& \leq C \sum_{j\in\Z} \sum_{k=-3}^3 \int_s^t \int_{\T^2}
		\| P_{j+k} v_d \|_{L^\infty_z}  \| \vp_{j+k} \ast v_1 \|_{L^2_z}
		\cdot \| \vp_j \ast ( \dH v_d )  \|_{L^2_z} \, d \xH d\tau \\
	&\leq C \sum_{j\in\Z} \sum_{k=-3}^3 2^{-k/2} \int_s^t \int_{\T^2}
		\| P_{j+k} v_d \|_{L^\infty_z}
		\| \vp_{j+k} \ast v_1 \|_{L^2_z}
		\cdot \| \vp_j \ast ( \dH v_d )  \|_{L^2_z} \, d \xH d\tau.
\end{aligned}
\end{align}

Here, H\"{o}lder's inequality and \autoref{prop:Sobolev} (2) yield that, for $2<\gamma<\infty$,
\begin{align}
\begin{aligned}\label{eq: J2_2}
	\| P_{j+k} v_d \|_{L^\infty_z}
	&\leq \sum_{l\leq j+k-3} \| \vp_l \ast v_d \|_{L^\infty_z} \\
	&\leq C \sum_{l\leq j+k-3} 2^l \| \vp_l \ast v_d \|_{L^1_z} \\
	&\leq C \sum_{l\leq j+k-3} 2^l \| \vp_l \ast v_d \|_{L^2_z} \qquad \text{(H\"{o}lder's inequality)} \\
	&\leq C \sum_{l\leq j+k-3} 2^{{2/\gamma} l} 
		\| \vp_l \ast v_d \|_{L^2_z}^{2/\gamma}
		\| \vp_l \ast (\dz v_d) \|_{L^2_z}^{1-2/\gamma} \\
	&\leq C 2^{(j+k)2/\gamma} \| v_d \|_{L^2_z}^{2/\gamma} \| \dz v_d \|_{L^2_z}^{1-2/\gamma}.
\end{aligned}
\end{align}

Therefore, we have from \eqref{eq: J2_2}, H\"{o}lder's inequality, and Minkowski's inequality that
\begin{align*}
	J_2 &\leq C \sum_{j\in\Z} \sum_{k=-3}^3 2^{-k/2} \int_s^t \int_{\T^2}
		\| v_d \|_{L^2_z}^{2/\gamma} \| \dz v_d \|_{L^2_z}^{1-2/\gamma} \\
		& \hspace{150pt}
		\cdot \left[ 2^{2(j+k)/\gamma} 
			\| \vp_{j+k} \ast v_1 \|_{L^2_z} \right]
		\cdot \| \vp_j \ast ( \dH v_d )  \|_{L^2_z} \, d \xH d\tau \\
	&\leq C \sum_{j\in\Z} \sum_{k=-3}^3 2^{-k/2} \int_s^t
			\| v_d \|_{L^2}^{2/\gamma} \| \dz v_d \|_{L^2}^{1-2/\gamma} \\
		&\hspace{150pt}
		\cdot \left[ 2^{2(j+k)/\gamma} 
			\left\| \| \vp_{j+k} \ast v_1 \|_{L^2_z} \right\|_{\LH{\infty}} \right]
		\cdot \| \vp_j \ast ( \dH v_d ) \|_{L^2} \, d\tau\\
	&\leq C \sum_{j\in\Z} \sum_{k=-3}^3 2^{-k/2} \int_s^t
			\| v_d \|_{L^2}^{2/\gamma} \| \dz v_d \|_{L^2}^{1-2/\gamma} \\
		&\hspace{150pt}
		\cdot \left[ 2^{2(j+k)/\gamma} 
			\| \vp_{j+k} \ast v_1 \|_{L^2_z\LH{\infty}} \right]
		\cdot \| \vp_j \ast ( \dH v_d )  \|_{L^2} \, d\tau.
\end{align*}

H\"{o}lder's inequality for number sequences, \autoref{prop:equiv}, and Young's inequality yield
\begin{align}
\begin{aligned}\label{eq: J2}
	J_2 &\leq C \sum_{k=-3}^3 2^{-k/2} \int_s^t
			\| v_d \|_{L^2}^{2/\gamma} \| \dz v_d \|_{L^2}^{1-2/\gamma}
		\cdot \| v_1 \|_{\dBrpqz{2/\gamma}{2}{2} \LH{\infty}}
		\cdot \left\| \| \dH v_d \|_{\dBrpqz{0}{2}{2}} \right\|_{\LH{2}} \,d\tau \\
	&\leq C \int_s^t
			\| v_d \|_{L^2}^{2/\gamma} \| \dz v_d \|_{L^2}^{1-2/\gamma}
		\cdot \| v_1 \|_{\dBrpqz{2/\gamma}{2}{2} \LH{\infty}}
		\cdot \| \dH v_d \|_{L^2} \,d\tau \\
	&\leq \frac{1}{6}\int_s^t \| \nabla v_d\|_{L^2}^2 \,d\tau 
		+ C \int_s^t
		\| v_d \|_{L^2}^2 
		\cdot \| v_1 \|_{\dBrpqz{2/\gamma}{2}{2} \LH{\infty}}^{\gamma} \,d\tau.
\end{aligned}
\end{align}
\medskip
\noindent{\it \underline{Step 3.} Estimate of $J_3$.} 

\begin{align}
\begin{aligned}\label{eq: J3_1}
J_3 &= 
	\sum_{j\in\Z} \sum_{ \substack{ \max{ \{ k, l \} } \\ \qquad  \geq  j-3 }  } 
		\sum_{ |k-l| \leq 2 } \int_s^t \int_{\T^3} 
	\left[ \vpt_j \ast \left( ( \vp_{k} \ast v_d ) \otimes ( \vp_{l} \ast v_1 ) \right) \right]
		\colon \left[ \vp_j \ast ( \dH v_d ) \right] \,dx d\tau \\	
	&\leq \sum_{j\in\Z} \sum_{ \substack{ \max{ \{ k, l \} } \\ \qquad  \geq  j-3 }  } 
		\sum_{ |k-l| \leq 2 } \int_s^t \int_{\T^2}
		\| \vpt_j \ast \left( ( \vp_{k} \ast v_d ) \otimes ( \vp_{l} \ast v_1 ) \right) \|_{L_z^{q'}} 
		\| \vp_j \ast ( \dH v_d ) \|_{L_z^q} \, d \xH d\tau \\
	&\leq \sum_{j\in\Z} \sum_{ k \geq  j-5 } 
		\sum_{ l = -2 }^{2} \int_s^t \int_{\T^2}
		\| \vpt_j \ast \left( ( \vp_{k} \ast v_d ) \otimes ( \vp_{k+l} \ast v_1 ) \right) \|_{L_z^{q'}} 
		\| \vp_j \ast ( \dH v_d ) \|_{L_z^q} \, d \xH d\tau \\
	&\leq C \sum_{j\in\Z} \sum_{ k \geq  j-5 } \sum_{ l = -2 }^{2} \int_s^t \int_{\T^2} 
		\| \vp_{k} \ast v_d \|_{L^2_z}
		\| \vp_{k+l} \ast v_1 \|_{L_p^p} 
		\| \vp_j \ast ( \dH v_d ) \|_{L_z^q} \, d \xH d\tau \\
	&= C \sum_{j\in\Z} \sum_{ k \geq -5 } \sum_{ l = -2 }^{2} \int_s^t \int_{\T^2} 
		\| \vp_{j+k} \ast v_d \|_{L^2_z}  
		\| \vp_{j+k+l} \ast v_1 \|_{L_z^p} 
		\| \vp_j \ast ( \dH v_d ) \|_{L_z^q} \, d \xH d\tau
\end{aligned}
\end{align}
Here, we have 
\begin{align}
\begin{aligned}\label{eq: J3_2}
	\| \vp_{j+k} \ast v_d \|_{L^2_z} 
	&= \| \vp_{j+k} \ast v_d \|_{L^2_z}^{1-2/p} 
		\| \vp_{j+k} \ast v_d \|_{L^2_z}^{2/p} \\
	&\leq 2^{ - (j+k) \cdot 2/p } \| \vp_{j+k} \ast v_d \|_{L^2_z}^{1-2/p} 
		\left[ 2^{j+k} \| \vp_{j+k} \ast v_d \|_{L^2_z} \right]^{2/p} \\
	&\leq C \cdot  2^{ - (j+k) \cdot 2/p } \| \vp_{j+k} \ast v_d \|_{L^2_z}^{1-2/p}
		\| \vp_{j+k} \ast ( \dz v_d ) \|_{L^2_z}^{2/p}.
\end{aligned}
\end{align}
By \autoref{prop:Sobolev} (note that $1/2-1/q=1/p$), 
\begin{equation}\label{eq: J3_3}
	\| \vp_j \ast ( \dH v_d ) \|_{L_z^q}
	\leq  C 2^{j/p} \| \vp_j \ast ( \dH v_d ) \|_{L^2_z}.
\end{equation}
Applying \eqref{eq: J3_2}, \eqref{eq: J3_3} and Minkowski's inequality, we have
\begin{align*}
	J_3
	&\leq C \sum_{j\in\Z} \sum_{ k \geq -5 } \sum_{ l = -2 }^{2} \int_s^t \int_G 
		2^{ -(j+k)\cdot 2/p }\| \vp_{j+k} \ast v_d \|_{L^2_z}^{1-2/p}  
		\| \vp_{j+k} \ast ( \dz v_d ) \|_{L^2_z}^{2/p} \\
	&\hspace{150pt}
		\cdot \| \vp_{j+k+l} \ast v_1 \|_{L_z^p} 
		\cdot 2^{j/p} \| \vp_j \ast ( \dH v_d ) \|_{L_z^2} \, d \xH d\tau \\
	&\leq C \sum_{j\in\Z} \sum_{ k \geq -5 } 2^{-k/p} 
		\sum_{ l = -2 }^{2} 2^{l/p}
		\int_s^t 
		\| \vp_{j+k} \ast v_d \|_{L^2}^{1-2/p}
		\| \vp_{j+k} \ast ( \dz v_d ) \|_{L^2}^{2/p} \\
	&\hspace{150pt}
		\cdot \left[ 2^{ -(j+k+l)/p } \| \vp_{j+k+l} \ast v_1 \|_{ L_z^p \LH{\infty} } \right]
		\cdot \| \vp_j \ast ( \dH v_d ) \|_{L^2} d\tau \\
	&\leq C  \sum_{j\in\Z} \sum_{ k \geq -5 } 2^{-k/p} 
		\int_s^t 
		\| \vp_{j+k} \ast v_d \|_{L^2}^{1-2/p}
		\| \vp_{j+k} \ast ( \dz v_d ) \|_{L^2}^{2/p}
		\| \vp_j \ast ( \dH v_d ) \|_{L_2}
		\| v_1 \|_{ \dBrpqz{-1/p}{p}{\infty}\LH{\infty} }.
\end{align*}
Let $\disp C_p = \sum_{k\geq-5} 2^{-k/p} \in (0, \infty)$, then again Young's inequality and \autoref{prop:equiv} yield that
\begin{align}
\begin{aligned}\label{eq: J3}
	J_3
	& \leq \sum_{ k \geq -5 } 2^{-k/p} \left( 
		\sum_{j\in\Z} 
		\int_s^t \frac{1}{6c^2C_p} \| \vp_j \ast ( \dH v_d ) \|_{L_2}^2 \,d\tau
		+ \sum_{j\in\Z} 
			\int_s^t \frac{1}{6c^2C_p} \| \vp_{j+k} \ast ( \dz v_d ) \|_{L^2}^2
		\right. \\
	& \hspace{250pt}
		\left.
		+ C \sum_{j\in\Z} \int_s^t \| \vp_{j+k} \ast v_d \|_{L^2}^2
			\| v_1 \|_{ \dBrpqz{-1/p}{p}{\infty}\LH{\infty} }^\beta \,d\tau
		\right) \\
	&= \sum_{ k \geq -5 } 2^{-k/p} \left( 
		\sum_{j\in\Z} 
		\int_s^t \frac{1}{6c^2C_p} \| \vp_j \ast ( \dH v_d ) \|_{L_2}^2 \,d\tau
		+ \sum_{j\in\Z} 
			\int_s^t \frac{1}{6c^2C_p} \| \vp_{j} \ast ( \dz v_d ) \|_{L^2}^2
		\right. \\
	& \hspace{250pt}
		\left.
		+ C \sum_{j\in\Z} \int_s^t \| \vp_{j} \ast v_d \|_{L^2}^2
			\| v_1 \|_{ \dBrpqz{-1/p}{p}{\infty}\LH{\infty} }^\beta \,d\tau
		\right) \\
	&\leq \frac{1}{6c^2} \sum_{j\in\Z} \int_s^t \| \vp_j \ast ( \nabla v_d ) \|_{L^2}^2 \,d\tau 
		+ C \sum_{j\in\Z} \int_s^t \| \vp_{j} \ast v_d \|_{L^2}^2
			\| v_1 \|_{ \dBrpqz{-1/p}{p}{\infty}\LH{\infty} }^\beta \,d\tau \\
	&= \frac{1}{6c^2} \int_s^t 
		\left\| \| \nabla v_d \|_{ \dBrpqH{0}{2}{2} } \right\|_{L^2_z}^2 \,d\tau 
		+ C \int_s^t \left\| \| v_d \|_{ \dBrpqH{0}{2}{2} } \right\|_{L^2}^2
			\| v_1 \|_{ \dBrpqz{-1/p}{p}{\infty}\LH{\infty} }^\beta \,d\tau \\
	&\leq \frac{1}{6} \int_s^t 
		\| \nabla v_d \|_{L^2}^2 \,d\tau 
		+ C \int_s^t \| v_d \|_{L^2}^2
			\| v_1 \|_{ \dBrpqz{-1/p}{p}{\infty}\LH{\infty} }^\beta \,d\tau.
\end{aligned}
\end{align}

\bigskip
\noindent{\it \underline{Step 4.} Estimate of $J_4$.} 

\begin{align}
\begin{aligned}\label{eq: J4_1}
J_4 &= 
	\sum_{j\in\Z} \sum_{k=-3}^3 \int_s^t \int_{\T^3}
	\left[ \vpt_j \ast \left( ( \vp_{j+k} \ast w_d ) ( P_{j+k}v_1 ) \right) \right]
		\cdot \left[ \vp_j \ast ( \dz v_d ) \right] \,dx d\tau \\	
	& \leq \sum_{j\in\Z} \sum_{k=-3}^3 \int_s^t \int_{\T^2}
		\| \vpt_j \ast \left( ( \vp_{j+k} \ast w_d ) ( P_{j+k}v_1 ) \right) \|_{L^{q'}_z}
		\| \vp_j \ast ( \dz v_d ) \|_{L^q_z} \, d \xH d\tau \\
	& \leq C \sum_{j\in\Z} \sum_{k=-3}^3 \int_s^t \int_{\T^2}
		\| \vp_{j+k} \ast w_d \|_{L^2_z}  \| P_{j+k}v_1 \|_{L^{p}_z}
		\cdot 2^j \| \vp_j \ast v_d  \|_{L^q_z} \, d \xH d\tau \\
	&\leq C \sum_{j\in\Z} \sum_{k=-3}^3 \int_s^t \int_{\T^2}
		\| \vp_{j+k} \ast ( \dz w_d) \|_{L^2_z} \left[ \sum_{l \leq j+k-3}  \| \vp_l \ast v_1 \|_{L^{p}_z} \right]
		\cdot  \| \vp_j \ast v_d \|_{L^q_z} \, d \xH d\tau \\
	&\leq C \sum_{j\in\Z} \sum_{k=-3}^3 \int_s^t 
		\| \vp_{j+k} \ast ( \divH{v_d}) \|_{L^2} 
		\left\| \sum_{l \leq j+k-3}  { \| \vp_l \ast v_1 \|_{L^p_z} } \right\|_{ \LH{\infty} }
		\cdot  \left\| \| \vp_j \ast v_d \|_{L^q_z} \right\|_{ \LH{2} } \, d\tau .
\end{aligned}
\end{align}

Applying \eqref{eq: Lq-interpolation} and \eqref{eq: v1-lowfreq} for \eqref{eq: J3_1}, we have
\begin{align*}
	J_4 
	&\leq C \sum_{j\in\Z} \sum_{k=-3}^3 \int_s^t 
		\| \vp_{j+k} \ast ( \div_H{v_d} ) \|_{L^2} 
		\cdot 2^{(j+k)/p} \|v_1\|_{ \dBrpqz{-1/p}{p}{\infty}\LH{\infty} } \\
		& \hspace{150pt}
		\cdot 2^{-j/p} \| \vp_j \ast v_d \|_{L^2}^{1-2/p} \| \vp_j \ast (\dz v_d) \|_{L^2}^{2/p} \, d\tau \\
	&\leq C \sum_{j\in\Z} \sum_{k=-3}^3 \int_s^t 
		\| \vp_{j+k} \ast ( \dH{v_d} ) \|_{L^2} 
		\cdot \| \vp_j \ast (\dz v_d) \|_{L^2}^{2/p}
		\cdot \| \vp_j \ast v_d \|_{L^2}^{1-2/p}
		\cdot \|v_1\|_{ \dBrpqz{-1/p}{p}{\infty}\LH{\infty} } \, d\tau.
\end{align*}

It follows again by the same calculation with \eqref{eq: J1} for $J_1$ that
\begin{align}
\begin{aligned}\label{eq: J4}
	J_4 
	&\leq \frac{1}{6} \int_s^t \| v_d \|_{L^2}^2 \, d\tau
	+ C \int_s^t \| v_d \|_{L^2}^2
		 \|v_1\|_{ \dBrpqz{-1/p}{p}{\infty}\LH{\infty} }^\beta \, d\tau.
\end{aligned}
\end{align}

\medskip
\noindent{\it \underline{Step 5.} Estimate of $J_5$.} 

\begin{align}
\begin{aligned}\label{eq: J5_1}
	J_5 &= 
	\sum_{j\in\Z} \sum_{k=-3}^3 \int_s^t \int_{\T^3}
	\left[ \vpt_j \ast \left( ( P_{j+k}w_d ) ( \vp_{j+k} \ast v_1 ) \right) \right]
		\cdot \left[ \vp_j \ast ( \dz v_d ) \right] \,dx d\tau \\	
	& \leq \sum_{j\in\Z} \sum_{k=-3}^3 \int_s^t \int_{\T^2}
		\| \vpt_j \ast \left( ( P_{j+k}w_d ) ( \vp_{j+k} \ast v_1 ) \right) \|_{L^2_z}
		\| \vp_j \ast ( \dz v_d ) \|_{L^2_z} \, d \xH d\tau \\
	& \leq C \sum_{j\in\Z} \sum_{k=-3}^3 \int_s^t \int_{\T^2}
		\| P_{j+k}w_d \|_{L^\infty_z}  \| \vp_{j+k} \ast v_1 \|_{L^2_z} \\
		&\hspace{150pt}
		\cdot 2^{2j/\gamma} \| \vp_j \ast v_d  \|_{L^2_z}^{2/\gamma}
		\cdot \| \vp_j \ast ( \dz v_d ) \|_{L^2_z}^{1-2/\gamma} \, d \xH d\tau.
\end{aligned}
\end{align}
Here, 
\begin{align*}
	\| P_{j+k} w_d \|_{L^\infty_z}
	&\leq C \| w_d \|_{L^\infty_z} \\
	&\leq C \| \divH v_d \|_{L^1_z} \qquad \text{(def.~of $w$)} \\
	&\leq C \| \divH v_d \|_{L^2_z} \qquad \text{(H\"{o}lder's inequality)} \\
	&\leq C \| \dH v_d \|_{L^2_z}.
\end{align*}
Therefore, we have
\begin{align}
\begin{aligned}\label{eq: J5}
	J_5 &\leq 
	C \sum_{j\in\Z} \sum_{k=-3}^3 2^{-2k/\gamma} \int_s^t \int_G
		\| \dH v_d \|_{L^2_z} 
		\cdot \left[ 2^{2(j+k)/\gamma}  \| \vp_{j+k} \ast v_1 \|_{L^2_z} \right] \\
		&\hspace{200pt}
		\cdot \| \vp_j \ast v_d  \|_{L^2_z}^{2/\gamma}
		\cdot \| \vp_j \ast ( \dz v_d ) \|_{L^2_z}^{1-2/\gamma} \, d \xH d\tau \\
	&\leq 
	C \sum_{j\in\Z} \sum_{k=-3}^3 2^{-2k/\gamma} \int_s^t
		\| \dH v_d \|_{L^2} 
		\cdot \left[ 2^{2(j+k)/\gamma}
			 \left\| \| \vp_{j+k} \ast v_1 \|_{L^2_z} \right\|_{\LH{\infty}} 
			 \right] \\
		&\hspace{200pt}	
		\cdot \| \vp_j \ast v_d  \|_{L^2}^{2/\gamma}
		\cdot \| \vp_j \ast ( \dz v_d ) \|_{L^2}^{1-2/\gamma} \, d\tau \\
	&\leq 
	C \sum_{j\in\Z} \sum_{k=-3}^3 2^{-2k/\gamma} \int_s^t
		\| \dH v_d \|_{L^2} 
		\cdot \left[ 2^{2(j+k)/\gamma}
			\| \vp_{j+k} \ast v_1 \|_{L^2_z\LH{\infty}}
			\right] \\
		&\hspace{200pt}
		\cdot \| \vp_j \ast v_d  \|_{L^2}^{2/\gamma}
		\cdot \| \vp_j \ast ( \dz v_d ) \|_{L^2}^{1-2/\gamma} \, d\tau \\
	&\leq
	C \sum_{k=-3}^3 2^{-2k/\gamma} \int_s^t
		\| \dH v_d \|_{L^2} 
		\cdot \| v_1 \|_{\dBrpq{2/\gamma}{2}{2}\LH{\infty}} 
		\left\| \| v_d \|_{\dBrpqz{0}{2}{2}} \right\|_{\LH{2}}^{2/\gamma}
		\cdot \left\| \| \dz v_d \|_{\dBrpqz{0}{2}{2}} \right\|_{\LH{2}}^{1-2/\gamma} \, d\tau \\
	&\leq C \int_s^t
		\| \dH v_d \|_{L^2} 
		\cdot \| v_1 \|_{\dBrpqz{2/\gamma}{2}{2}\LH{\infty}}
		\| v_d \|_{L^2}^{2/\gamma}
		\| \dz v_d \|_{L^2}^{1-2/\gamma} \, d\tau \\
	&\leq \frac{1}{6} \int_s^t \| \nabla v_d \|_{L^2}^2 \, d\tau
		+ C \int_s^t \| v_d \|_{L^2}^2
			 \| v_1 \|_{\dBrpq{2/\gamma}{2}{2}\LH{\infty}}^{\gamma} \, d\tau.
\end{aligned}
\end{align}

\medskip
\noindent{\it \underline{Step 6.} Estimate of $J_6$.} 

\begin{align}
\begin{aligned}\label{eq: J6_1}
J_6 &= 
	\sum_{j\in\Z} \sum_{ \substack{ \max{ \{ k, l \} } \\ \qquad  \geq  j-3 }  } 
		\sum_{ |k-l| \leq 2 } \int_s^t \int_{\T^3} 
	\left[ \vpt_j \ast \left( ( \vp_{k} \ast w_d ) \otimes ( \vp_{l} \ast v_1 ) \right) \right]
		\colon \left[ \vp_j \ast ( \dz v_d ) \right] \,dx d\tau \\	
	&\leq \sum_{j\in\Z} \sum_{ k \geq  j-5 } 
		\sum_{ l = -2 }^{2} \int_s^t \int_{\T^2}
		\| \vpt_j \ast \left( ( \vp_{k} \ast w_d ) \otimes ( \vp_{k+l} \ast v_1 ) \right) \|_{L_z^{q'}} 
		\| \vp_j \ast ( \dz v_d ) \|_{L_z^q} \, d \xH d\tau \\
	&\leq C \sum_{j\in\Z} \sum_{ k \geq  j-5 } \sum_{ l = -2 }^{2} \int_s^t \int_{\T^2} 
		\| \vp_{k} \ast w_d \|_{L^2_z}
		\| \vp_{k+l} \ast v_1 \|_{L_p^p} 
		\cdot 2^j \| \vp_j \ast v_d \|_{L_z^q} \, d \xH d\tau \\
	&\leq C \sum_{j\in\Z} \sum_{ k \geq -5 } 2^{-k}
		\sum_{ l = -2 }^{2} \int_s^t \int_{\T^2} 
		\| \vp_{j+k} \ast (\dz w_d) \|_{L^2_z}  
		\| \vp_{j+k+l} \ast v_1 \|_{L_z^p} 
		\| \vp_j \ast v_d \|_{L_z^q} \, d \xH d\tau \\
	&\leq C \sum_{j\in\Z} \sum_{ k \geq -5 } 2^{-k}
		\sum_{ l = -2 }^{2} \int_s^t 
		\| \vp_{j+k} \ast (\divH v_d) \|_{L^2}  
		\left\| \| \vp_{j+k+l} \ast v_1 \|_{L_z^p} \right\|_{\LH{\infty}}
		\left\| \| \vp_j \ast v_d \|_{L_z^q} \right\|_{ \LH{2} } \, d\tau .
\end{aligned}
\end{align}
Applying Minkowski's inequality and \eqref{eq: Lq-interpolation} for \eqref{eq: J6_1}, we have
\begin{align*}
J_6 &\leq C
	\sum_{j\in\Z} \sum_{ k \geq -5 } 2^{-(1-1/p)k}
		\sum_{ l = -2 }^{2} 2^{l/p} \int_s^t 
		\| \vp_{j+k} \ast (\divH v_d) \|_{L^2}  \\
		&\hspace{100pt}
		\cdot \left[ 2^{-(j+k+l)/p} \| \vp_{j+k+l} \ast v_1 \|_{L_z^p \LH{\infty} } \right] 
		\| \vp_j \ast v_d \|_{L^2}^{1-2/p} \| \vp_j \ast ( \dz v_d ) \|_{L^2}^{2/p} \, d\tau \\
	&\leq C
	\sum_{j\in\Z} \sum_{ k \geq -5 } 2^{-(1-1/p)k}
		\int_s^t 
		\| \vp_{j+k} \ast (\divH v_d) \|_{L^2} \\
		&\hspace{100pt}
		\cdot \| v_1 \|_{\dBrpqz{-1/p}{p}{\infty}\LH{\infty}} 
		\| \vp_j \ast v_d \|_{L^2}^{1-2/p} \| \vp_j \ast ( \dz v_d ) \|_{L^2}^{2/p} \, d\tau.
\end{align*}
In a similar manner as \eqref{eq: J3}, we have
\begin{equation}\label{eq: J6}
	J_6 \leq \frac{1}{6} \int_s^t 
		\| \nabla v_d \|_{L^2}^2 \,d\tau 
		+ C \int_s^t \| v_d \|_{L^2}^2
			\| v_1 \|_{ \dBrpqz{-1/p}{p}{\infty}\LH{\infty} }^\beta \,d\tau.
\end{equation}
Collecting the estimates \eqref{eq: J1}, \eqref{eq: J2}, \eqref{eq: J3}, \eqref{eq: J4}, \eqref{eq: J5} and \eqref{eq: J6}, we finally obtain \eqref{eq: desired}:
\[
J_1 + \cdots + J_6
	\leq \int_s^t \| \nabla v_d (\tau) \|_{L^2}^2	
	+ C \int_s^t \| v_d (\tau) \|_{L^2}^2 
		\left( \| v_1 (\tau) \|_{ \dBrpqz{-1/p}{p}{\infty}\LH{\infty} }^\beta
		+ \| v_1 \|_{\dBrpq{2/\gamma}{2}{2}\LH{\infty}}^{\gamma}  \right) \, d\tau.
\qed 
\]

\subsection{2D case}\label{subsec:uniqueness 2D}
In this \autoref{subsec:uniqueness 2D}, we use the same notation $v_d = v_1 - v_2$ for weak solutions $v_j=(v_j^1, v_j^2)$ ($j=1, 2$).
Corresponding to \autoref{lem:1st form}, we have to introduce the following lemma which is a consequence of \cite[eq.~(4.8)]{J2021}.
\begin{lem}\label{lem:1st form 2D}
Let $v_1$ and $v_2$ be weak solutions to the two-dimensional primitive equations \eqref{eq:PE2} with the energy inequality.
\par
Then, we have
\begin{align}
\begin{aligned}\label{eq:1st form 2D}
	&\, \frac{1}{2}\| v_d (t) \|_{L^2}^2 - \frac{1}{2}\| v_d(s) \|_{L^2}^2
	+ \int_s^t \| \nabla v_d (\tau) \|^2_{L^2} \,d \tau \\
	=&\, \int_s^t \int_{\T^2} ( {}^\top(v^1_d, w_d) \otimes v_1 )(x, \tau) \colon \nabla_{x_1} v_d (x, \tau)  \,d x d \tau \\
	=&\, \int_s^t \int_{\T^2} v_d^1(x, \tau)\, v_1(x, \tau) \cdot \p_{x_1} v_d (x, \tau)  \,d x d \tau 
		+ \int_s^t \int_{\T^2}  w_d(x, \tau)\, v_1(x, \tau) \cdot \dz v_d (x, \tau)  \,d x d \tau
\end{aligned}
\end{align}
for $0<s\leq t\leq T$.
\end{lem}

As in \eqref{eq:RHS 1st form}, we have from \autoref{lem:1st form 2D} that the following representation holds:
\begin{align*}
	&\, \text{R.H.S.~ of \eqref{eq:1st form 2D}} \\ 
	=&\, \sum_{j\in\Z} \sum_{k=-3}^3 \int_s^t \int_{\T^2}
	\left[ \vpt_j \ast \left( ( \vp_{j+k} \ast v_d^1 ) ( P_{j+k}v_1 ) \right) \right]
		\cdot \left[ \vp_j \ast ( \p_{x_1} v_d ) \right] \,dx d\tau \\ 
	&\quad + \sum_{j\in\Z} \sum_{k=-3}^3 \int_s^t \int_{\T^2}
	\left[ \vpt_j \ast \left( ( P_{j+k} v_d^1 ) ( \vp_{j+k}\ast v_1 ) \right) \right]
		\cdot \left[ \vp_j \ast ( \p_{x_2} v_d ) \right] \,dx d\tau \\ 
	&\quad +\sum_{j\in\Z} \sum_{ \substack{ \max{ \{ k, l \} } \\ \qquad  \geq  j-3 }  } 
		\sum_{ |k-l| \leq 2 } \int_s^t \int_{\T^2}
	\left[ \vpt_j \ast \left( ( \vp_{k} \ast v_d^1 ) ( \vp_{l} \ast v_1 ) \right) \right] 
		\cdot \left[ \vp_j \ast ( \p_{x_1} v_d ) \right] \,dx d\tau \\	
	&\quad + \sum_{j\in\Z} \sum_{k=-3}^3 \int_s^t \int_{\T^2}
	\left[ \vpt_j \ast \left( ( \vp_{j+k} \ast w_d ) ( P_{j+k}v_1 ) \right) \right]
		\cdot \left[ \vp_j \ast ( \dz v_d ) \right] \,dx d\tau \\ 
	&\quad + \sum_{j\in\Z} \sum_{k=-3}^3 \int_s^t \int_{\T^2}
	\left[ \vpt_j \ast \left( ( P_{j+k}w_d ) ( \vp_{j+k} \ast v_1 ) \right) \right]
		\cdot \left[ \vp_j \ast ( \dz v_d ) \right] \,dx d\tau \\ 
	&\quad + \sum_{j\in\Z} \sum_{ \substack{ \max{ \{ k, l \} } \\ \qquad  \geq  j-3 }  } 
		\sum_{ |k-l| \leq 2 } \int_s^t \int_{\T^2}
	\left[ \vpt_j \ast \left( ( \vp_{k} \ast w_d ) ( \vp_{l} \ast v_1 ) \right) \right]
		\colon \left[ \vp_j \ast ( \dz v_d ) \right] \,dx d\tau. 
\end{align*}
Therefore, under the assumption \eqref{eq:asp}, we obtain the uniqueness result in the same manner as the proof for the three-dimensional case. \qed

\section{Proof of the energy equality}
\label{sec:proof of EE}

\subsection{3D case}
In order to prove \autoref{thm:energy eq}, let
\begin{equation}\label{eq: vN}
v_{\leq N} 
	\equiv \psi_{N} \ast_{z} \chi_{N} \ast_{\H} v 
	=  \sum_{j, j' \leq N}\vp_{j} \ast_{z} \theta_{j'} \ast_{\H} v,
\end{equation}
where $\{\vp_j\}_{j\in\Z}$ (resp.~$\{\theta_j\}_{j\in\Z}$) is the  {one} -dimensional (resp.~ {two} -dimensional) Littlewood--Paley decomposition of the unity and
\[
	\psi_N = \sum_{j\leq N} \vp_j \qquad \left( \text{resp.}\ \chi_N = \sum_{j\leq N} \theta_j  \right).
\]
In this section, $\ast_z$ (resp.~$\ast_\H$) denotes the  {one} -dimensional (resp.~ {two} -dimensional) convolution in $z$-direction (resp.~$x_\H$-direction).

  For the three-dimensional case, choosing $\phi = \left( v_{\leq N} \right)_{\leq N}$ in \eqref{eq:weak solution PE}  , we have
\begin{align}
\begin{aligned}\label{eq:energy eq error}
	&\, \frac{1}{2}\| v_{\leq N}(t) \|_{L^2}^2
		- \frac{1}{2}\| ( v_0 )_{\leq N} \|_{L^2}^2
		+ \int_0^t \int_{\T^3} |\nabla v_{\leq N}|^2\,dx d\tau \\
	=&\, \int_0^t \int_{\T^3} \left[ (v\otimes v)_{\leq N} 
		- v_{\leq N} \otimes v_{\leq N} \right]: \dH v_{\leq N} \, dx d\tau
		+ \int_0^t \int_{\T^3} \left[ (w v)_{\leq N} 
		- w_{\leq N} v_{\leq N} \right]\cdot \dz v_{\leq N} \, dx d\tau .
\end{aligned}
\end{align}
We define functions $I_j$ ($1\leq j\leq 8$) as follows:
\begin{align*}
	I_1(x, t) &= 
		\left[ 
		\chi_N \ast_{\H} \left( 
		\int_{\R} 
		\psi(\zeta) { \left( v(\cdot, z-2^{-N}\zeta, t) -v(\cdot, z, t) \right)
			\otimes \left( v(\cdot, z-2^{-N}\zeta, t) -v(\cdot, z, t) \right) } \,d\zeta 
		\right) 
		\right] (x_\H); \\
	I_2(x, t) &=
		- \left[
		\chi_N \ast_H \left(
		\left( \psi_N \ast_z v - v \right)
		\otimes \left( \psi_N \ast_z v - v \right)
		\right)
		\right] (x, t); \\
	I_3(x, t) &=
		\int_{\R^2} \chi(y_\H) \, 
		\left[ 
		\psi_N \ast_z \left( v(x_\H - 2^{-N}y_\H, \cdot, t) - v(x_\H, \cdot, t) \right) \right] (z) \\
		& \hspace{150pt} 
		\otimes \left[ 
		\psi_N \ast_z \left( v(x_\H - 2^{-N}y_\H, \cdot , t) - v(x_\H, \cdot, t) \right) \right]  (z) \,dy_\H; \\
	I_4(x, t) &= 
		- \left[ 
		\left( 
		\psi_N \ast_z \left( \chi_N\ast_{\H} v - v \right)
		\right) \otimes \left(
		\psi_N \ast_z \left( \chi_N\ast_{\H} v - v \right)
		\right) \right] (x, t); \\
	I_5(x, t) &=
		\left[
		\chi_N \ast_{\H} \left( 
		\int_{\R} 
		\psi(\zeta) { \left( w(\cdot, z-2^{-N}\zeta, t) -w(z, t) \right)
			\left( v(\cdot, z-2^{-N}\zeta, t) -v(z, t) \right) } \,d\zeta 
		\right)
		\right] (x_\H); \\
	I_6(x, t) &=
		- \left[ 
		\chi_N \ast_H \left(
		\left( \psi_N \ast_z w - w \right)
		\left( \psi_N \ast_z v - v \right)
		\right) 
		\right] (x, t); \\
	I_7(x, t) &= 
		\int_{\R^2} \chi(y_\H) \, 
		\left[ 
		\psi_N \ast_z \left( w(x_\H - 2^{-N}y_\H, \cdot, t) - w(x_\H, \cdot, t) \right) \right] (z) \\
		& \hspace{150pt} 
		\cdot \left[ 
		\psi_N \ast_z \left( v(x_\H - 2^{-N}y_\H, \cdot, t) - v(x_\H, \cdot, t) \right) \right]  (z) \,dy_\H; \\
	I_8(x, t) &=
		- \left[ 
		\left( 
		\psi_N \ast_z \left( \chi_N\ast_{\H} w - w \right)
		\right) \left(
		\psi_N \ast_z \left( \chi_N\ast_{\H} v - v \right)
		\right) 
		\right] (x, t).
\end{align*}
Then, 
\begin{equation}\label{eq:error rhs}
	\text{R.H.S.~ of \eqref{eq:energy eq error}} \,
	= \sum_{m=1}^4 \int_0^t \int_{\T^3} I_m : \dH v_{\leq N} \,dx d\tau 
		+ \sum_{m=5}^8 \int_0^t \int_{\T^3} I_m \cdot \dz v_{\leq N} \,dx d\tau.
\end{equation}

We first show the following estimate.
\begin{lem}\label{lem:3-1}
Let $1\leq p\leq \infty$. Then, it holds that
\[
\| \dH v_{\leq N} \|_{L^\infty} \leq C 2^{ (1+2/p)N} \| v \|_{\dBrpqz{-1/p}{p}{\infty} \LH{\infty}}
\]
for any $v\in \dBrpqz{-1/p}{p}{\infty}\LH{\infty}$.
\end{lem}

\begin{proof}
Since
\[
	\| \dH v_{\leq N} \|_{\LH{\infty}} 
	= \left\| (\dH \chi_N )
		\ast_\H \left(
		\psi_N \ast_z v 
		\right) 
		\right\|_{\LH{\infty}}
	\leq C 2^N \| \psi_N \ast_z v  \|_{\LH{\infty}},
\]
we have
\begin{align*}
	\| \dH v_{\leq N} \|_{L^\infty} 
	&\leq C2^N \left\| 
		\| \psi_N \ast_z v \|_{L^\infty_z} 
		\right\|_{\LH{\infty}}.
\end{align*}
Since it holds by \autoref{prop:Sobolev} (2) that
\[
	\| \psi_N \ast_z v \|_{L^\infty_z}
	\leq \sum_{j\leq N} \| \vp_j \ast_z v \|_{L^\infty_z}
	\leq C \sum_{j\leq N} 2^{2j/p} \cdot \left[ 2^{-j/p}\| \vp_j \ast_z v \|_{L^p_z} \right],
\]
we have by Minkowski's inequality that
\begin{align*}
	\| \dH v_{\leq N} \|_{L^\infty} 
	&\leq C 2^N \left\| 
		\sum_{j\leq N} 2^{2j/p} \cdot \left[ 2^{-j/p}\| \vp_j \ast_z v \|_{L^p_z} \right]
		\right\|_{\LH{\infty}} \\
	&\leq C 2^N \cdot \sum_{j\leq N} 2^{2j/p} \cdot 
		\left[ 2^{-j/p}\| \vp_j \ast_z v \|_{\LH{\infty}(L^p_z)} \right] \\
	&\leq C 2^N \cdot \sum_{j\leq N} 2^{2j/p} \cdot 
		\left[ 2^{-j/p}\| \vp_j \ast_z v \|_{L^p_z(\LH{\infty})} \right] \\
	&\leq C 2^N \cdot \sum_{j\leq N} 2^{2j/p} \cdot 
		\| v \|_{\dBrpqz{-1/p}{p}{\infty}\LH{\infty}} \\
	&\leq C 2^{(1+2/p)N} \| v \|_{\dBrpqz{-1/p}{p}{\infty}\LH{\infty}}.
\end{align*}
\end{proof}

We list the essential estimates for $I_j$ ($j=1, 2, 3, 4$) to control the former part of the energy flux.

\begin{lem}\label{lem:3-2}
Let $2\leq p\leq \infty$. Assume that $v\in H^1(\T^3)$. Then, we have the following estimates.
\begin{enumerate}[(1)]
\item 
$
		\| I_1 \|_{L^1} 
		\leq C 2^{-(1+2/p)N} 
			\left( \int_{\R} |\psi(\zeta)| |\zeta|^{1+2/p} 
			\| v(\cdot, \cdot-2^{-N}\zeta) -v \|_{L^2}^{1-2/p} \,d\zeta \right) 
			\| \dz v \|_{L^2}^{1+2/p}.
	$
\item 
	$
		\| I_2 \|_{L^1} 
		\leq C 2^{-(1+2/p)N} 
			\| \psi_N \ast_z v - v \|_{L^2}^{1-2/p}
			\| \dz v \|_{L^2}^{1+2/p}.
	$
\item 
	$
		\| I_3 \|_{L^1} 
		\leq C 2^{-(1+2/p)N} 
			\left(
			\int_{\R^{n-1}} |\chi(y_\H)| |y_\H|^{1+2/p}
			\| v(\cdot - 2^{-N}y_\H, \cdot ) - v \|_{L^2}^{1-2/p}
			\,dy_\H
			\right) 
			\| \dH v \|_{L^2}^{1+2/p}.
	$
\item 
	$
		\| I_4 \|_{L^1} 
		\leq C 2^{-(1+2/p)N} 
			\| \chi_N\ast_{\H} v - v \|_{L^2}^{1-2/p}
			\| \dH v \|_{L^2}^{1+2/p}.
	$
\end{enumerate}
\end{lem}

\begin{proof}
The estimates (1)-(4) can be shown in a similar manner which Hausdorff--Young's inequality for convolution and the interpolation inequality are applied. The terms $I_j$ ($j=1, 2, 3, 4$) can be estimates as follows.
\begin{enumerate}[(1)]
\item \begin{align*}
	&\, \| I_1 \|_{L^1} \\
		=&\, \left\| \chi_N \ast_{\H} \left( 
		\int_{\R} 
		\psi(\zeta) { \left( v(\cdot, \cdot-2^{-N}\zeta) -v \right)
			\otimes \left( v(\cdot, \cdot-2^{-N}\zeta) -v \right) } \,d\zeta 
		\right) \right\|_{ \LH{1} L^{1}_z } \\
	\leq&\, C\, \int_{\R} |\psi(\zeta)| \cdot
		\left\|
		\| v(\cdot, \cdot-2^{-N}\zeta) -v \|_{L^2_z}^2
		\right\|_{\LH{1}} \,d\zeta \\
	=&\, C\, \int_{\R} |\psi(\zeta)| \cdot
		\left\|
		\| v(\cdot, \cdot-2^{-N}\zeta) -v \|_{L^2_z}^{1-2/p}
		\| v(\cdot, \cdot-2^{-N}\zeta) -v \|_{L^2_z}^{1+2/p}
		\right\|_{\LH{1}} \,d\zeta \\
	\leq&\, \int_{\R} |\psi(\zeta)| \cdot
		\left\|
		\| v(\cdot, \cdot-2^{-N}\zeta) -v \|_{L^2_z}^{1-2/p}
		\cdot \left[ 2^{-N}|\zeta| \| \dz v \|_{L^2_z} \right]^{1+2/p}
		\right\|_{\LH{1}} \,d\zeta \\
	\leq&\, C 2^{-(1+2/p)N} \left( \int_{\R} |\psi(\zeta)| |\zeta|^{1+2/p} \cdot 
		\| v(\cdot, \cdot-2^{-N}\zeta) -v \|_{L^2}^{1-2/p} \,d\zeta \right) 
		\cdot 
		\| \dz v \|_{L^2}^{1+2/p}.
\end{align*}

\item \begin{align*}
	\| I_2 \|_{L^1} 
		&= \left\| \chi_N \ast_H \left(
		\left( \psi_N \ast_z v - v \right)
		\otimes \left( \psi_N \ast_z v - v \right)
		\right) \right\|_{L^1} \\
	&\leq C \left\| \| \psi_N \ast_z v - v \|_{L^2_z}^2 \right\|_{\LH{1}} \\
	&= C \left\| \| \psi_N \ast_z v - v \|_{L^2_z}^{1-2/p}
		\| \psi_N \ast_z v - v \|_{L^2_z}^{1+2/p} \right\|_{\LH{1}} \\
	&\leq C \left\| \| \psi_N \ast_z v - v \|_{L^2_z}^{1-2/p}
		\cdot \left[ 2^{-N} \| \dz v \|_{L^2_z} \right]^{1+2/p} \right\|_{\LH{1}} \\
	&\leq C 2^{-(1+2/p)N} 
		\| \psi_N \ast_z v - v \|_{L^2}^{1-2/p}
		\| \dz v \|_{L^2}^{1+2/p}.
\end{align*}

\item \begin{align*}
	&\, \| I_3 \|_{L^1} \\
	=&\, \left\|
		\int_{\R^2} \chi(y_\H) \, 
		\left[ 
		\psi_N \ast_z \left( v(\cdot - 2^{-N}y_\H, \cdot ) - v\right) \right] 
		\otimes \left[ 
		\psi_N \ast_z \left( v(\cdot - 2^{-N}y_\H, \cdot ) - v \right) \right] \,dy_\H
		\right\|_{L^1} \\
	\leq &\, 
		\int_{\R^2} |\chi(y_\H)|
		\| v(\cdot - 2^{-N}y_\H, \cdot ) - v \|_{L^2}^2 \,dy_\H \\
	=&\, \int_{\R^2} |\chi(y_\H)|
		\| v(\cdot - 2^{-N}y_\H, \cdot ) - v \|_{L^2}^{1-2/p}
		\| v(\cdot - 2^{-N}y_\H, \cdot ) - v \|_{L^2}^{1+2/p} 
		\,dy_\H \\
	\leq&\, C \int_{\R^2} |\chi(y_\H)|
		\| v(\cdot - 2^{-N}y_\H, \cdot ) - v \|_{L^2}^{1-2/p}
		\cdot \left[ 2^{-N}|y_H| \| \dH v \|_{L^2} \right]^{1+2/p}
		\,dy_\H \\
	\leq&\, C 2^{-(1+2/p)N} \left(
		\int_{\R^2} |\chi(y_\H)| |y_\H|^{1+2/p}
		\| v(\cdot - 2^{-N}y_\H, \cdot ) - v \|_{L^2}^{1-2/p}
		\,dy_\H
		\right) \cdot \| \dH v \|_{L^2}^{1+2/p}.
\end{align*}

\item \begin{align*}
	\| I_4 \|_{L^1} 
	&= \left\| \left( \psi_N \ast_z \left( \chi_N\ast_{\H} v - v \right)
		\right) \otimes \left(
		\psi_N \ast_z \left( \chi_N\ast_{\H} v - v \right) \right) \right\|_{L^1} \\
	&\leq \left\|
		\psi_N \ast_z \left( \chi_N\ast_{\H} v - v \right)
		\right\|_{L^2}^2 \\
	&\leq C \| \chi_N\ast_{\H} v - v \|_{L^2}^2 \\
	&= C \| \chi_N\ast_{\H} v - v \|_{L^2}^{1-2/p}
		\| \chi_N\ast_{\H} v - v \|_{L^2}^{1+2/p} \\
	&\leq C 2^{-(1+2/p)N} \| \chi_N\ast_{\H} v - v \|_{L^2}^{1-2/p}
		\| \dH v \|_{L^2}^{1+2/p}.
		\qedhere 
\end{align*}
\end{enumerate}
\end{proof}

\begin{lem}\label{lem:3-3}
If $v\in L^\infty(0, T; L^2) \cap L^2(0, T; H^1) \cap L^\beta (0, T; \dBrpqz{-1/p}{p}{\infty}\LH{\infty})$ for some $2<\beta,\, p<\infty$ with $2/\beta+2/p=1$, then it holds that
\[
	\sum_{m=1}^4 \int_0^t \int_{\T^3} I_m : \dH v_{\leq N} \,dx d\tau  
	\to 0 \qquad \text{as $N\to\infty$}
\]
for $0<t\leq T$.
\end{lem}

\begin{proof}
\autoref{lem:3-1} and \ref{lem:3-2} yield
\begin{align*}
	&\, \left| \sum_{m=1}^4 \int_0^t \int_{\T^3} I_m : \dH v_{\leq N} \,dxd\tau \right| \\
	\leq &\, \sum_{m=1}^4
	\int_0^t \| I_m(\tau) \|_{L^1}
		\| \dH v_{\leq N} (\tau) \|_{L^\infty} \,d\tau \\
	\leq &\, C \int_0^t 
		\left( \int_{\R} |\psi(\zeta)| |\zeta|^{1+2/p} \cdot 
		\| v(\cdot, \cdot-2^{-N}\zeta) -v \|_{L^2}^{1-2/p} \,d\zeta \right) 
		\cdot 
		\| \dz v \|_{L^2}^{1+2/p}
		\| v \|_{\dBrpqz{-1/p}{p}{\infty}\LH{\infty}} \,d\tau \\
		&+ C \int_0^t 
		\| \psi_N \ast_z v - v \|_{L^2}^{1-2/p}
		\| \dz v \|_{L^2}^{1+2/p}
		\| v \|_{\dBrpqz{-1/p}{p}{\infty}\LH{\infty}} \,d\tau \\
		& + C \int_0^t \left(
		\int_{\R^2} |\chi(y_\H)| |y_\H|^{1+2/p}
		\| v(\cdot - 2^{-N}y_\H, \cdot ) - v \|_{L^2}^{1-2/p}
		\,dy_\H
		\right) \cdot \| \dH v \|_{L^2}^{1+2/p}
		\| v \|_{\dBrpqz{-1/p}{p}{\infty}\LH{\infty}} \,d\tau \\
		& + C \int_0^t 
		\| \chi_N\ast_{\H} v - v \|_{L^2}^{1-2/p}
		\| \dH v \|_{L^2}^{1+2/p}
		\cdot 
		\| v \|_{\dBrpqz{-1/p}{p}{\infty}\LH{\infty}} \,d\tau.
\end{align*}

By the assumption $v\in L^\infty(0, T; L^2) \cap L^2(0, T; H^1)$ and $v\in L^\beta(0, T; \dBrpqz{-1/p}{p}{\infty}\LH{\infty})$, we can apply the Lebesgue convergence theorem for each term above. The first term tends to 0 as $N\to \infty$ due to the continuity of the transition in $L^2_z(\T)$. The third term also tends to 0. The second term tends to 0 due to $\psi_N \ast_z v \to v$ in $L^2_z(\T)$. The forth term also tends to 0. \par
Consequently, we have that 
	\[
	\left| \sum_{m=1}^4 \int_0^t \int_{\T^3} I_m : \dH v_{\leq N} \,dxd\tau \right| \to 0. \qedhere 
	\]
\end{proof}

Next, we deal with the second term of the right hand side of \eqref{eq:energy eq error}. We introduce two types of the estimates of $\dz v_{\leq N}$.
\begin{lem}\label{lem:3-4}
\begin{enumerate}[(1)]
\item Let $1\leq p\leq \infty$. Then it holds that
	\[
	\| \dz v_{\leq N} \|_{\LH{\infty}L^p_z} 
	\leq C 2^{-(1+1/p)N} \|v\|_{\dBrpqz{-1/p}{p}{\infty}\LH{\infty}}
	\]
	for any $v\in \dBrpqz{-1/p}{p}{\infty}\LH{\infty}$.

\item Let $2\leq \gamma\leq \infty$. Then it holds that
	\[
	\| \dz v_{\leq N} \|_{\LH{\infty}L^2_z}
	\leq C 2^{-(1-2/\gamma)N} \| v \|_{\dBrpqz{2/\gamma}{2}{\infty}\LH{\infty}}
	\]
	for any $v \in \dBrpqz{2/\gamma}{2}{\infty}\LH{\infty}$.
\end{enumerate}
\end{lem}

\begin{proof}
The proof is straightforward and similar with that of \autoref{lem:3-1}, thus we omit it.
\end{proof}

\begin{lem}\label{lem:3-5}
Let $2\leq p \leq \infty$ and $1/p+1/{p'}=1$. Assume that $v\in H^1(\T^3)$. Then, the following estimates hold.
\begin{enumerate}[(1)]
\item 
$
	\| I_5 \|_{\LH{1}L^{p'}_z}
	\leq C 2^{-(1+1/p)N} \left( \int_{\R} |\psi(\zeta)| |\zeta|^{1+2/p} 
		\| v(\cdot, \cdot-2^{-N}\zeta) -v \|_{L^2}^{1-2/p} 
		\,d\zeta \right) 
		\| \nabla v \|_{L^2}^{1+2/p}.
	$

\item 
	$
	\| I_6 \|_{\LH{1}L^{p'}_z}
	\leq C 2^{-(1+1/p)N} 
		\| \psi_N \ast_z v - v \|_{L^2}^{1-2/p}
		\| \nabla v \|_{L^2}^{1+2/p} .
	$
\end{enumerate}
\end{lem}

\begin{proof}
\begin{enumerate}[(1)]
\item We can estimate $\| I_5 \|_{ \LH{1} L^{p'}_z }$ by Hausdorff--Young's inequality for convolution, Minkowski's inequality and H\"{o}lder's inequality as
	\begin{align}
	\begin{aligned}\label{eq:I5_1}
	\| I_5 \|_{ \LH{1} L^{p'}_z }
	&= \left\| \chi_N \ast_{\H} \left( 
		\int_{\R} 
		\psi(\zeta) { \left( w(\cdot, \cdot-2^{-N}\zeta) -w \right)
			\left( v(\cdot, \cdot-2^{-N}\zeta) -v \right) } \,d\zeta 
		\right) \right\|_{ \LH{1} L^{p'}_z } \\
	&\leq C\, \int_{\R} |\psi(\zeta)| \cdot
		\left\|
		\| w(\cdot, \cdot-2^{-N}\zeta) -w \|_{L^2_z}
		\| v(\cdot, \cdot-2^{-N}\zeta) -v \|_{L^q_z}
		\right\|_{\LH{1}} \,d\zeta.
\end{aligned}
\end{align}
Since
\[
	\| w(\cdot, \cdot-2^{-N}\zeta) -w \|_{L^2_z}
	\leq C|\zeta| \cdot 2^{-N} \| \dz w \|_{L^2_z}
	= C|\zeta| \cdot 2^{-N} \| \divH v \|_{L^2_z},
\]
we have
\begin{equation}\label{eq:I5_2}
	\| w(\cdot, \cdot-2^{-N}\zeta) -w \|_{L^2}
	\leq C |\zeta| \cdot 2^{-N} \| \divH v \|_{L^2}.
\end{equation}
Since we can estimate as
\begin{align*}
	\| v(\cdot, \cdot-2^{-N}\zeta) -v \|_{L^q_z}
	&\leq \| v(\cdot, \cdot-2^{-N}\zeta) -v \|_{L^2_z}^{1-2/p}
		\| v(\cdot, \cdot-2^{-N}\zeta) -v \|_{L^\infty_z}^{2/p} \\
	&\leq C \| v(\cdot, \cdot-2^{-N}\zeta) -v \|_{L^2_z}^{1-2/p}
		\left[ 2^{-N/2}|\zeta| \|v\|_{\dot{C}^{1/2} } \right]^{2/p} \\
	&\leq C |\zeta|^{2/p} \cdot 2^{-N/p} 
		\| v(\cdot, \cdot-2^{-N}\zeta) -v \|_{L^2_z}^{1-2/p}
		\|\dz v\|_{L^2_z}^{2/p},
\end{align*}
we have
\begin{equation}\label{eq:I5_3}
	\| v(\cdot, \cdot-2^{-N}\zeta) -v \|_{\LH{2} L^q_z}
	\leq C |\zeta|^{2/p} \cdot 2^{-N/p} 
		\| v(\cdot, \cdot-2^{-N}\zeta) -v \|_{L^2}^{1-2/p}
		\|\dz v\|_{L^2}^{2/p}.
\end{equation}
Consequently, it follows from \eqref{eq:I5_1}, \eqref{eq:I5_2} and \eqref{eq:I5_3} that
\begin{align*}
	&\, \| I_5 \|_{\LH{1}L^{p'}_z} \\
	\leq&\,
		C 2^{-(1+1/p)N} \left( \int_{\R} |\psi(\zeta)| |\zeta|^{1+2/p} \cdot 
		\| v(\cdot, \cdot-2^{-N}\zeta) -v \|_{L^2}^{1-2/p} 
		\,d\zeta \right)
		\| \divH v \|_{L^2} \|\dz v\|_{L^2}^{2/p}  \\
	\leq&\,
		C 2^{-(1+1/p)N} \left( \int_{\R} |\psi(\zeta)| |\zeta|^{1+2/p} \cdot 
		\| v(\cdot, \cdot-2^{-N}\zeta) -v \|_{L^2}^{1-2/p} 
		\,d\zeta \right) 
		\| \nabla v \|_{L^2}^{1+2/p}.
\end{align*}
	
\item Hausdorff--Young's inequality for convolution and H\"{o}lder's inequality yield
	\begin{align}
	\begin{aligned}\label{eq:I6_1}
	\| I_6 \|_{\LH{1}L^{p'}}
	&= \left\|  \chi_N \ast_H \left(
		\left( \psi_N \ast_z w - w \right)
		\left( \psi_N \ast_z v - v \right)
		\right) \right\|_{\LH{1}L^{p'}_z} \\
	&\leq C \left\|
		\| \psi_N \ast_z w - w \|_{L^2_z}
		\| \psi_N \ast_z v - v \|_{L^q_z}
		\right\|_{\LH{1}}.
\end{aligned}
\end{align}
Since 
\[ 
	\| \psi_N \ast_z w - w \|_{L^2_z} 
	\leq C 2^{-N} \| \dz w \|_{L^2_z} 
	= C2^{-N} \| \divH v \|_{L^2_z},
\]
we have
\begin{equation}\label{eq:I6_2}
	\| \psi_N \ast_z w - w \|_{\LH{2}L^2_z} 
	\leq C2^{-N}  \| \divH v \|_{L^2}.
\end{equation}
Since it holds that
\begin{align*}
	\| \psi_N \ast_z v - v \|_{L^q_z}
	&\leq \| \psi_N \ast_z v - v \|_{L^2_z}^{1-2/p} \| \psi_N \ast_z v - v \|_{L^\infty_z}^{2/p} \\
	&\leq C \| \psi_N \ast_z v - v \|_{L^2_z}^{1-2/p}
		\left[ \| \psi_N \ast_z v - v \|_{L^\infty_z} \right]^{2/p},
\end{align*}
we have
\begin{align*}
	\| \psi_N \ast_z v - v \|_{L^\infty_z}
	&\leq \sum_{j>N} \| \vp_j \ast_z v \|_{L^\infty_z} \\
	&\leq C \sum_{j>N} 2^{-j/2} \| \vp_j \ast_z (\dz v) \|_{L^2_z} \\
	&\leq C 2^{-N/2} \| \dz v \|_{L^2_z}.
\end{align*}
Therefore, we have
\begin{align}
\begin{aligned}\label{eq:I6_3}
	\left\| \| \psi_N \ast_z v - v \|_{L^q_z} \right\|_{\LH{2}}
	&\leq C 2^{-(1+1/p)N} \| \divH v \|_{L^2} \| \dz v \|_{L^2}^{2/p} 
		\| \psi_N \ast_z v - v \|_{L^2_z}^{1-2/p} \\
	&\leq C 2^{-(1+1/p)N} \| \nabla v \|_{L^2}^{1+2/p} 
		\| \psi_N \ast_z v - v \|_{L^2}^{1-2/p}.
\end{aligned}
\end{align}
Consequently, it follows from \eqref{eq:I6_1}, \eqref{eq:I6_2} and \eqref{eq:I6_3} that
\begin{align*}
	\| I_6 \|_{\LH{1}L^{p'}_z}
	&\leq C 2^{-(1+1/p)N} \| \nabla v \|_{L^2}^{1+2/p} 
		\| \psi_N \ast_z v - v \|_{L^2}^{1-2/p}. \qedhere 
\end{align*}
\end{enumerate}
\end{proof}

\begin{lem}\label{lem:3-6}
Let $2\leq \gamma \leq \infty$. Assume that $v\in H^1(\T^3)$. Then, the following estimates hold.
\begin{enumerate}[(1)]
\item $
	\| I_7 \|_{\LH{1}L^{2}_z}
	\leq C 2^{-(1-2/\gamma) N} \left( 
		\int_{\R^2} |\chi(y_\H)| |y_\H|^{1-2/\gamma} 
		\| v(\cdot - 2^{-N}y_\H, \cdot ) - v \|_{L^2}^{2/\gamma} \,dy_\H
		\right) \cdot
		\| \dH v \|_{L^2}^{2-2/\gamma}.
	$
\item $
	\| I_8 \|_{\LH{1}L^2_z}
	\leq C2^{-(1-2/\gamma)N} \| \chi_N\ast_{\H} v - v \|_{L^2}^{2/\gamma}
		\| \dH v \|_{L^2}^{2-2/\gamma}.
	$
\end{enumerate}
\end{lem}

\begin{proof}
\begin{enumerate}[(1)]
\item We can estimate $\| I_7 \|_{ \LH{1} L^2_z }$ by Minkowski's inequality and H\"{o}lder's inequality as
	\begin{align*}
	&\, \| I_7 \|_{\LH{1}L^2_z} \\
	=&\, \left\|
	\int_{\R^2} \chi(y_\H) \, 
		\left[ 
		\psi_N \ast_z \left( w(\cdot - 2^{-N}y_\H, \cdot ) - w \right) \right]
		\cdot \left[ 
		\psi_N \ast_z \left( v(\cdot - 2^{-N}y_\H, \cdot ) - v \right) \right] 
	\right\|_{\LH{1}L^2_z} \\
	\leq &\, 
	\int_{\R^2} | \chi(y_\H) |
		\| \psi_N \ast_z \left( w(\cdot - 2^{-N}y_\H, \cdot ) - w \right) \|_{\LH{2}L^\infty_z}
		\left\|
		\| \psi_N \ast_z \left( v(\cdot - 2^{-N}y_\H, \cdot ) - v \right) \|_{L^2_z}
		\right\|_{\LH{2}} \,dy_\H.
\end{align*}
Here, it holds that
\begin{align*}
	&\, \| \psi_N \ast_z \left( w(\cdot - 2^{-N}y_\H, \cdot) - w \right) \|_{\LH{2} L^\infty_z} \\
	\leq &\, C \left\| \| w\|_{\LH{\infty}} \right\|_{\LH{2}} \\
	\leq &\, C \left\|  \int_{-\pi}^{\pi} |\divH{v}|\,d\zeta \right\|_{\LH{2}} \\
	\leq &\, C \| \divH{v} \|_{L^2} \qquad \left( L^2_z(\T) \hookrightarrow L^1_z(\T) \right) \\
	\leq &\, C \| \dH{v} \|_{L^2},
\end{align*}
and
\begin{align*}
	&\, 
	\left\|
		\| \psi_N \ast_z \left( v(\cdot - 2^{-N}y_\H, \cdot ) - v \right) \|_{L^2_z}
		\right\|_{\LH{2}} \\
	\leq&\, C \| v(\cdot - 2^{-N}y_\H, \cdot ) - v \|_{L^2} \\
	=&\, C \| v(\cdot - 2^{-N}y_\H, \cdot ) - v \|_{L^2}^{2/\gamma}
		\left\| \| v(\cdot - 2^{-N}y_\H, \cdot ) - v \|_{\LH{2}} \right\|_{L^2_z}^{1-2/\gamma} \\
	\leq&\, C \| v(\cdot - 2^{-N}y_\H, \cdot ) - v \|_{L^2}^{2/\gamma}
		\cdot \left[ 2^{-N}|y_\H| \left\| \| \dH v \|_{\LH{2}} \right\|_{L^2_z} \right]^{1-2/\gamma} \\
	=&\, C |y_\H|^{1-2/\gamma} \cdot 2^{-(1-2/\gamma)N} \| v(\cdot - 2^{-N}y_\H, \cdot ) - v \|_{L^2}^{2/\gamma}
		\| \dH v \|_{L^2}^{1-2/\gamma}.
\end{align*}
Therefore, we have
\begin{align*}
	\| I_7 \|_{\LH{1}L_z^{2}} 
	\leq\, C 2^{-(1-2/\gamma)N} \left( 
		\int_{\R^2} |\chi(y_\H)| |y_\H|^{1-2/\gamma} 
		\| v(\cdot - 2^{-N}y_\H, \cdot ) - v \|_{L^2}^{2/\gamma} \,dy_\H
		\right) \cdot
		\| \dH v \|_{L^2}^{2-2/\gamma}.
\end{align*}

\item We can estimate $\| I_8 \|_{ \LH{1} L^{p'}_z }$ by H\"{o}lder's inequality as
\begin{align*}
	&\, \left\| \left( 
		\psi_N \ast_z \left( \chi_N\ast_{\H} w - w \right) 
		\right) \left(
		\psi_N \ast_z \left( \chi_N\ast_{\H} v - v \right) 
		\right) \right\|_{\LH{1}L^{2}_z} \\
	\leq &\, 
		\| \psi_N \ast_z \left( \chi_N\ast_{\H} w - w \right) \|_{\LH{2} L^\infty_z}
		\| \psi_N \ast_z \left( \chi_N\ast_{\H} v - v \right) \|_{L^2}.
\end{align*}
Here, it holds that
\begin{align*}
	&\, \| \psi_N \ast_z \left( \chi_N\ast_{\H} w - w \right) \|_{\LH{2} L^\infty_z} \\
	\leq &\, C \left\| \| w\|_{\LH{\infty}} \right\|_{\LH{2}} \\
	\leq &\, C \left\|  \int_{-\pi}^{\pi} |\divH{v}|\,d\zeta \right\|_{\LH{2}} \\
	\leq &\, C \| \divH{v} \|_{L^2} \qquad \left( L^2_z(\T) \hookrightarrow L^1_z(\T) \right) \\
	\leq &\, C \| \dH{v} \|_{L^2},
\end{align*}
and
\begin{align*}
	&\, \| \psi_N \ast_z \left( \chi_N\ast_{\H} v - v \right) \|_{L^2} \\
	= & \| \psi_N \ast_z \left( \chi_N\ast_{\H} v - v \right) \|_{L^2}^{2/\gamma}
		 \| \psi_N \ast_z \left( \chi_N\ast_{\H} v - v \right) \|_{L^2}^{1-2/\gamma} \\
	\leq &\, C \| \chi_N\ast_{\H} v - v \|_{L^2}^{2/\gamma}
		\cdot \left[ 2^{-N} \left\| \| \dH v \|_{L^2_z} \right\|_{\LH{2}} \right]^{1-2/\gamma} \\
	\leq &\, C 2^{-(1-2/\gamma)N} \| \chi_N\ast_{\H} v - v \|_{L^2}^{2/\gamma}
		\| \dH v \|_{L^2}^{1-2/\gamma}.
\end{align*}

Therefore, we have
\[
	\| I_8 \|_{\LH{1}L^2_z}
	\leq C 2^{-(1-2/\gamma)N} \| \chi_N\ast_{\H} v - v \|_{L^2}^{2/\gamma} \| \dH v \|_{L^2}^{2-2/\gamma}.\qedhere
\] 
\end{enumerate}
\end{proof}

\begin{lem}\label{lem:3-7}
If $v\in L^\infty(0, T; L^2) \cap L^2(0, T; H^1) \cap L^\beta (0, T; \dBrpqz{-1/p}{p}{\infty}\LH{\infty}) \cap L^\gamma(0, T; \dBrpqz{2/\gamma}{2}{\infty}\LH{\infty})$ for some $2<\beta,\, p<\infty$ with $2/\beta + 2/p = 1$ and for some $2<\gamma<\infty$, then it holds that

\[
	\sum_{m=5}^8 \int_0^t \int_{\T^3} I_m \cdot \dz v_{\leq N} \,dx d\tau  
	\to 0 \qquad \text{as $N\to\infty$}
\]
for $0< t \leq T$.
\end{lem}

\begin{proof}
Applying H\"{o}lder's inequality, \autoref{lem:3-4} (1) and \autoref{lem:3-5}, we have
\begin{align*}
	&\, \left| \int_0^t \int_{\T^3} I_5 \cdot \dz v_{\leq N} \,dxd\tau
		+ \int_0^t \int_{\T^3} I_6 \cdot \dz v_{\leq N} \,dxd\tau \right| \\
	\leq &\, 
		\int_0^t \| I_5(\tau) \|_{\LH{1}L^{p'}_z}
		\| \dz v_{\leq N} (\tau) \|_{\LH{\infty}L^p_z} \,d\tau 
		+ \int_0^t \| I_6(\tau) \|_{\LH{1}L^{p'}_z}
		\| \dz v_{\leq N} (\tau) \|_{\LH{\infty}L^p_z} \,d\tau\\
	\leq &\, C \int_0^t 
		 \left( \int_{\R} |\psi(\zeta)| |\zeta|^{1+2/p} 
		\| v(\cdot, \cdot-2^{-N}\zeta) -v \|_{L^2}^{1-2/p} 
		\,d\zeta \right)
		\| \nabla v \|_{L^2}^{1+2/p}
		\| v \|_{\dBrpqz{-1/p}{p}{\infty}\LH{\infty}} \,d\tau \\
		&+ C \int_0^t 
		\| \psi_N \ast_z v - v \|_{L^2}^{1-2/p}
		\| \nabla v \|_{L^2}^{1+2/p}
		\| v \|_{\dBrpqz{-1/p}{p}{\infty}\LH{\infty}} \,d\tau .
\end{align*}

By the assumption $v\in L^\infty(0, T; L^2) \cap L^2(0, T; H^1)$ and $v\in L^\beta (0, T; \dBrpqz{-1/p}{p}{\infty}\LH{\infty})$, we can apply the Lebesgue convergence theorem for each term above. The first term tends to 0 as $N\to \infty$ due to the continuity of the transition in $L^2_z(\T)$. The second term tends to 0 due to $\psi_N \ast_z v \to v$ in $L^2_z(\T)$. The forth term also tends to 0. 
Therefore, we have that 
	\begin{equation}\label{eq:lem3-7-1}
	\left| \int_0^t \int_{\T^3} I_5 : \dz v_{\leq N} \,dxd\tau
		+ \int_0^t \int_{\T^3} I_6 : \dz v_{\leq N} \,dxd\tau \right|
		\to 0.
	\end{equation}

On the other hand, applying H\"{o}lder's inequality, \autoref{lem:3-4} (2) and \autoref{lem:3-6}, we have
\begin{align*}
	&\, \left| \int_0^t \int_{\T^3} I_7 \cdot \dz v_{\leq N} \,dxd\tau
		+ \int_0^t \int_{\T^3} I_8 \cdot \dz v_{\leq N} \,dxd\tau \right| \\
	\leq &\, 
		\int_0^t \| I_7(\tau) \|_{\LH{1}L^2_z}
		\| \dz v_{\leq N} (\tau) \|_{\LH{\infty}L^2_z} \,d\tau 
		+ \int_0^t \| I_8(\tau) \|_{\LH{1}L^2_z}
		\| \dz v_{\leq N} (\tau) \|_{\LH{\infty}L^2_z} \,d\tau\\
	\leq &\, C \int_0^t 
		 \left( 
		\int_{\R^2} |\chi(y_\H)| |y_\H|^{1-2/\gamma} 
		\| v(\cdot - 2^{-N}y_\H, \cdot ) - v \|_{L^2}^{2/\gamma} \,dy_\H
		\right)
		\| \dH v \|_{L^2}^{2-2/\gamma}
		\| v \|_{\dBrpqz{2/\gamma}{2}{\infty}\LH{\infty}} \,d\tau \\
		&+ C \int_0^t 
		\|  \chi_N\ast_{\H} v - v \|_{L^2}^{2/\gamma}
		\| \dH v \|_{L^2}^{2-2/\gamma}
		\| v \|_{\dBrpqz{2/\gamma}{2}{\infty}\LH{\infty}} \,d\tau .
\end{align*}

As in \eqref{eq:lem3-7-1}, we have from the assumptions $v\in L^\infty(0, T; L^2(\T^3)) \cap L^2(0, T; H^1(\T^3))$ and $v\in L^\gamma(0, T; \dBrpqz{2/\gamma}{2}{\infty}\LH{\infty})$ that
\[
	\left| \int_0^t \int_{\T^3} I_7 \cdot \dz v_{\leq N} \,dxd\tau
		+ \int_0^t \int_{\T^3} I_8 \cdot \dz v_{\leq N} \,dxd\tau \right| 
	\to 0.
\]
\end{proof}

\begin{proof}[Proof of \autoref{thm:energy eq}]
Since $v_{\leq N}$ and $\nabla v_{\leq N}$ approximate $v$ and $\nabla v$ in the sense of $L^2(\T^3)$, respectively, we can see that
\begin{align}
\begin{aligned}\label{eq:prf ee lhs}
	\text{L.H.S.~of \eqref{eq:energy eq error}}\,
	&= \frac{1}{2}\| v_{\leq N}(t) \|_{L^2}^2
		- \frac{1}{2}\| ( v_0 )_{\leq N} \|_{L^2}^2
		+ \int_0^t \int_{\T^3} |\nabla v_{\leq N}|^2\,dx d\tau \\
	&\to \frac{1}{2}\| v(t) \|_{L^2}^2
		- \frac{1}{2}\| v_0 \|_{L^2}^2
		+ \int_0^t \int_{\T^3} |\nabla v|^2\,dx d\tau
	\qquad \text{as $N\to\infty$}
\end{aligned}
\end{align}
with the aid of the assumption $v\in L^\infty(0, T; L^2(\T^3)) \cap L^2(0, T; H^1(\T^3))$.
Applying \autoref{lem:3-3} and \ref{lem:3-7} for \eqref{eq:error rhs}, we have
\begin{align}
\begin{aligned}\label{eq:prf ee rhs}
	\text{R.H.S.~of \eqref{eq:energy eq error}}\, 
	&= \sum_{m=1}^4 \int_0^t \int_{\T^3} I_m : \dH v_{\leq N} \,dx d\tau 
			+ \sum_{m=5}^8 \int_0^t \int_{\T^3} I_m \cdot \dz v_{\leq N} \,dx d\tau \\
	&\to 0 
	\qquad \text{as $N\to\infty$}.
\end{aligned}
\end{align}
Consequently, \eqref{eq:prf ee lhs} and \eqref{eq:prf ee rhs} yield the energy equality \eqref{eq:energy equality}.
\end{proof}

\subsection{2D case}
In order to prove \autoref{thm:energy eq}, let
\begin{equation}\label{eq: vN2}
v_{\leq N} 
	\equiv \psi_{N} \ast_{z} \widetilde{\chi_{N}} \ast_{x_1} v 
	=  \sum_{j, j' \leq N}\vp_{j} \ast_{z} \widetilde{\theta_{j'}} \ast_{x_1} v,
\end{equation}
where $\{\vp_j\}_{j\in\Z}$ (resp.~$\left\{\widetilde{\theta_j}\right\}_{j\in\Z}$) is the one-dimensional Littlewood--Paley decomposition in $z$-direction (resp.~$x_1$-ditrection) of the unity and
\[
	\psi_N = \sum_{j\leq N} \vp_j 
	\qquad \left( \text{resp.}\ \widetilde{\chi_N} = \sum_{j\leq N} \widetilde{\theta_j}  \right).
\]
In this section, $\ast_z$ (resp.~$\ast_{x_1}$) denotes the one-dimensional convolution in $z$-direction (resp.~$x_1$ direction). \par
As in three-dimensional case, choosing $\phi = \left( v_{\leq N} \right)_{\leq N}$ which is given in \eqref{eq: vN2} in \eqref{eq:weak solution PE2}, we have
\begin{align}
\begin{aligned}\label{eq:energy eq error 2D}
	&\, \frac{1}{2}\| v_{\leq N}(t) \|_{L^2}^2
		- \frac{1}{2}\| ( v_0 )_{\leq N} \|_{L^2}^2
		+ \int_0^t \int_{\T^2} |\nabla v_{\leq N}|^2\,dx d\tau \\
	=&\, \int_0^t \int_{\T^2} \left[ (v^1 v)_{\leq N} 
		- (v^1)_{\leq N} v_{\leq N} \right] \cdot \p_{x_1} v_{\leq N} \, dx d\tau
		+ \int_0^t \int_{\T^2} \left[ (w v)_{\leq N} 
		- w_{\leq N} v_{\leq N} \right]\cdot \dz v_{\leq N} \, dx d\tau \\
	=&\, \sum_{m=1}^4 \int_0^t \int_{\T^2} \widetilde{I_m} \cdot \p_{x_1} v_{\leq N} \,dx d\tau 
		+ \sum_{m=5}^8 \int_0^t \int_{\T^2} \widetilde{I_m} \cdot \dz v_{\leq N} \,dx d\tau,
\end{aligned}
\end{align}
where $\widetilde{I_1}, \widetilde{I_2}, \ldots, \widetilde{I_8}$ are given by
\begin{align*}
	\widetilde{I_1}(x, t) &= 
		\left[ 
		\widetilde{\chi_N} \ast_{x_1} \left( 
		\int_{\R} 
		\psi(\zeta) { \left( v^1(\cdot, z-2^{-N}\zeta, t) -v^1(\cdot, z, t) \right)
			\left( v(\cdot, z-2^{-N}\zeta, t) -v(\cdot, z, t) \right) } \,d\zeta 
		\right) 
		\right] (x_1); \\
	\widetilde{I_2} (x, t) &=
		- \left[
		\widetilde{\chi_N} \ast_{x_1} \left(
		\left( \psi_N \ast_z v^1 - v^1 \right)
		\left( \psi_N \ast_z v - v \right)
		\right)
		\right] (x, t); \\
	\widetilde{I_3}(x, t) &=
		\int_{\R} \widetilde{\chi}(y_1) \, 
		\left[ 
		\psi_N \ast_z \left( v^1(x_1 - 2^{-N}y_1, \cdot, t) - v^1(x_1, \cdot, t) \right) \right] (z) \\
		& \hspace{150pt} 
		\cdot \left[ 
		\psi_N \ast_z \left( v(x_1 - 2^{-N}y_1, \cdot , t) - v(x_1, \cdot, t) \right) \right]  (z) \,dy_1; \\
	\widetilde{I_4}(x, t) &= 
		- \left[ 
		\left( 
		\psi_N \ast_z \left( \widetilde{\chi_N}\ast_{x_1} v^1 - v^1 \right)
		\right) \left(
		\psi_N \ast_z \left( \widetilde{\chi_N}\ast_{x_1} v - v \right)
		\right) \right] (x, t); \\
	\widetilde{I_5}(x, t) &=
		\left[
		\widetilde{\chi_N} \ast_{x_1} \left( 
		\int_{\R} 
		\psi(\zeta) { \left( w(\cdot, z-2^{-N}\zeta, t) -w(z, t) \right)
			\left( v(\cdot, z-2^{-N}\zeta, t) -v(z, t) \right) } \,d\zeta 
		\right)
		\right] (x_1); \\
	\widetilde{I_6}(x, t) &=
		- \left[ 
		\widetilde{\chi_N} \ast_{x_1} \left(
		\left( \psi_N \ast_z w - w \right)
		\left( \psi_N \ast_z v - v \right)
		\right) 
		\right] (x, t); \\
	\widetilde{I_7}(x, t) &= 
		\int_{\R} \widetilde{\chi}(y_1) \, 
		\left[ 
		\psi_N \ast_z \left( w(x_1 - 2^{-N}y_1, \cdot, t) - w(x_1, \cdot, t) \right) \right] (z) \\
		& \hspace{150pt} 
		\cdot \left[ 
		\psi_N \ast_z \left( v(x_1 - 2^{-N}y_1, \cdot, t) - v(x_1, \cdot, t) \right) \right]  (z) \,dy_1; \\
	\widetilde{I_8}(x, t) &=
		- \left[ 
		\left( 
		\psi_N \ast_z \left( \widetilde{\chi_N}\ast_{x_1} w - w \right)
		\right) \left(
		\psi_N \ast_z \left( \widetilde{\chi_N}\ast_{x_1} v - v \right)
		\right) 
		\right] (x, t).
\end{align*}
Therefore, we also obtain the energy equality in the same manner as the proof for the three-dimensional case. \qed

\bigskip 

\noindent{\bf Acknowledgments.} The authors express their sincere thanks to Professor Hideo Kozono for his encouragement.

This project starts when the first-named author visited Waseda University of
Tokyo. He is grateful to the Department of Mathematics for its kind hospitality during this time.

This work was supported by JST SPRING, Grant Number JPMJSP2128. 

\bibliography{ref}
\bibliographystyle{abbrv}

\end{document}